\documentclass[11pt]{amsart}
\usepackage{amsmath}
\theoremstyle{plain}
\newtheorem{theorem}{Theorem}[section]
\newtheorem{lemma}[theorem]{Lemma}

\theoremstyle{definition}
\newtheorem{remark}[theorem]{Remark}
\newtheorem{definition}[theorem]{Definition}

\newtheorem{example}[theorem]{Example}
\numberwithin{equation}{section}

\def\be{\begin{equation}}
\def\ee{\end{equation}}

\begin{document}

\title[The Loewner-Nirenberg Problem in Singular Domains]
{The Loewner-Nirenberg Problem in Singular Domains}
\author[Han]{Qing Han}
\address{Beijing International Center for Mathematical Research\\
Peking University\\
Beijing, 100871, China} \email{qhan@math.pku.edu.cn}
\address{Department of Mathematics\\
University of Notre Dame\\
Notre Dame, IN 46556} \email{qhan@nd.edu}
\author[Shen]{Weiming Shen}
\address{School of Mathematical Sciences\\
Peking University\\
Beijing, 100871, China}
\email{wmshen@pku.edu.cn}
\address{Beijing International Center for Mathematical Research\\
Peking University\\
Beijing, 100871, China}

\begin{abstract}
We study the asymptotic behaviors of solutions of the Loewner-Nirenberg problem
in singular domains and prove that the solutions are well approximated by the
corresponding solutions in tangent cones at singular points on the boundary.
The conformal structure of the underlying equation plays an essential role in the derivation of
the optimal estimates.
\end{abstract}

\thanks{The first author acknowledges the support of NSF
Grant DMS-1404596. }
\maketitle

\section{Introduction}\label{sec-Intro}

Assume $\Omega\subset \mathbb{R}^{n}$ is a domain, for  $n\ge 3$. We consider the following
problem:
 \begin{align}
\label{eq-MainEq} \Delta u  &= \frac14n(n-2) u^{\frac{n+2}{n-2}} \quad\text{in }\Omega,\\
\label{eq-MainBoundary}u&=\infty\quad\text{on }\partial \Omega.
\end{align}
This is the so-called Loewner-Nirenberg problem,
also known as the singular Yamabe problem.
For a large class of domains $\Omega$, \eqref{eq-MainEq} and \eqref{eq-MainBoundary}
admit a unique positive solution $u\in C^\infty(\Omega)$.
Geometrically,  $u^{\frac{4}{n-2}}\sum_{i=1}^{n}dx_i\otimes dx_i$ is
a complete metric with the constant scalar curvature $-n(n-1)$ on $\Omega$.

The two dimensional counterpart
is given by, for $\Omega\subset\mathbb R^2$,
\begin{align}
\label{eq-MainEq-HigherDim} \Delta{u} = e^{ 2u } \quad\text{in }\Omega.
\end{align}
More generally, we can study, for a function $f$,
\begin{align*}
\label{eq-MainEq-HigherDim-general} \Delta{u} = f(u)\quad\text{in }\Omega.
\end{align*}

For bounded domains $\Omega$, let $d$ be the distance function to $\partial\Omega$.
If $\partial\Omega$ is $C^2$, then $d$ is
a $C^2$-function near $\partial\Omega$. In a pioneering work,
Loewner and Nirenberg \cite{Loewner&Nirenberg1974} studied asymptotic behaviors
of solutions of \eqref{eq-MainEq}  and \eqref{eq-MainBoundary} and proved, for $d$ sufficiently small,
\begin{equation}\label{eq-EstimateDegree1}|d^{\frac{n-2}{2}}u-1|\le Cd, \end{equation}
where $C$ is a positive constant depending only on certain geometric quantities of $\partial\Omega$.
This follows from a comparison of $u$ and the corresponding solutions
in the interior and exterior tangent balls.
This result has been generalized to more
general $f$ and up to higher order terms, for example, by
Brandle and Marcus \cite{Brandle&Marcus1995},
Diaz and Letelier \cite{DiazLetelier1992}, and  Kichenassamy \cite{Kichenassamy2005JFA}.
All these results require $\partial\Omega$ to have
some degree of regularity. The case where $\partial\Omega$ is singular was
studied by del Pino and Letelier \cite{delPino2002},
and Marcus and Veron \cite{Marcus&Veron1997}.
However, there are no explicit estimates in neighborhoods of singular boundary points in these works.

In \cite{HanShen}, we studied the asymptotic behaviors of solutions of \eqref{eq-MainEq-HigherDim}
and \eqref{eq-MainBoundary}
in planar singular domains, and proved that these solutions are well approximated by the
corresponding solutions in tangent cones near isolated singular points on the boundary.
Based on a combination of conformal transforms and the maximum principle, we derived an
optimal estimate.

In this paper, we study the
asymptotic behaviors of solutions of \eqref{eq-MainEq}  and \eqref{eq-MainBoundary}
near singular points on $\partial\Omega$. Similarly as in
\cite{HanShen}, we prove that the solutions
of \eqref{eq-MainEq}  and \eqref{eq-MainBoundary} can be approximated by the
corresponding solutions in tangent cones at singular points on the boundary.

Presumably, it is more difficult to discuss solutions of
\eqref{eq-MainEq}  and \eqref{eq-MainBoundary} for $n\ge 3$ than those
of \eqref{eq-MainEq-HigherDim}  and \eqref{eq-MainBoundary} for $n=2$, for several reasons.
First, the conformal invariance of domains is more restrictive for $n\ge3$. For example,
cones are not conformal to each other unless they are conjugate.
Second, there are no explicit solutions of \eqref{eq-MainEq}  and \eqref{eq-MainBoundary}
in cones in general. Third, the type of the boundary singularity is more diverse.
We need to introduce new techniques to address these issues.

Our main result in this paper is given by the following theorem.

\begin{theorem}\label{main reslut}
Let $\Omega\subset \mathbb{R}^{n}$ be a bounded Lipschitz domain with
$x_0\in\partial\Omega$ and, for some integer $k\le n$,  let $\partial\Omega$ in a neighborhood of
$x_0$ consist of
$k$ $C^{1,1}$-hypersurfaces $S_1, \cdots, S_k$ intersecting at $x_0$ with the property that the normal vectors
of $S_1, \cdots, S_k$ at $x_0$ are linearly independent. Suppose $ u \in
C^{\infty}(\Omega)$ is  a solution of \eqref{eq-MainEq}-\eqref{eq-MainBoundary},
and $u_{V_{x_0}}$ is the corresponding solution in the tangent cone $V_{x_0}$
of $\Omega$ at $x_0$.
Then, there exist a constant $r$ and a 
$C^{1,1}$-diffeomorphism $T$:
$B_r (x_0)\rightarrow T(B_r(x_0))\subseteq\mathbb R^n$, with
$T(\Omega \bigcap B_r (x_0))= V_{x_0}\bigcap T(B_r(x_0))$ and
$T(\partial\Omega \bigcap B_r (x_0))= \partial V_{x_0}\bigcap T(B_r(x_0))$, such that,
for any $x\in B_{r/2}(x_0)$,
\begin{align}
\label{main-estimate} \left|\frac{u(x)}{u_{V_{x_0}}(Tx)}-1\right|\leq C|x-x_0|,
\end{align}
where $C$ is a positive constant depending only on $n$
and the geometry of $\partial\Omega$.
\end{theorem}

The estimate \eqref{main-estimate} generalizes \eqref{eq-EstimateDegree1}
to singular domains and is optimal. The power one of the
distance in the right-hand side cannot be improved without better regularity
assumptions of the boundary. This estimate resembles a similar estimate for the 
equation for the positive scalar curvature near isolated singular points, established in 
\cite{Caffarelli1989}. A refined version was proved in \cite{Korevaar1999}. 
Refer to \cite{HanZ2010} for more general equations. 

In the proof of Theorem \ref{main reslut}, we will construct the map $T$, which is determined
by the distances to $S_i$.
The concept of tangent cones will be introduced in Section \ref{sec-Domains}.

We now describe briefly the proof of Theorem \ref{main reslut}, 
which is based on a combination of affine transforms, conformal transforms
and the maximum principle. To make a comparison, we note that, if the domain $\Omega$ is $C^{1,1}$, 
we can place an interior tangent ball and an exterior tangent ball at each of the boundary point 
and then compare the solution in $\Omega$ with the solution in the interior tangent ball and with the solution 
outside the exterior tangent ball. Refer to the proof of Theorem \ref{thrm-C-1,alpha-expansion}
for details. 
Now we assume that $\partial\Omega$ near a boundary point $x_0$ consists of $k$ $C^{1,1}$-hypersurfaces 
$S_1, \cdots, S_k$ intersecting at $x_0$, for some $k\le n$. 
Then, the tangent planes $P_1, \cdots, P_k$ of $S_1, \cdots, S_k$ 
at $x_0$ naturally bound a cone $V_{x_0}$, which is the tangent cone of $\Omega$ at $x_0$. 
Our goal is to compare the solution $u$ in $\Omega$ near $x_0$ 
with the solution $u_{V_{x_0}}$ in $V_{x_0}$. We note that 
a given point $x$ in $\Omega$ may not necessarily be a point in the tangent cone $V_{x_0}$. So as 
a part of the comparison of $u$ with $u_{V_{x_0}}$, we need to construct a map $T$, which maps 
$\Omega$ near $x_0$ onto $V_{x_0}$ near $x_0$, and to compare $u(x)$ with 
$u_{V_{x_0}}(Tx)$. We achieve this in two steps. 

In the first step, we construct two sets $\widetilde B$ and $\widehat B$ with the property 
$\widetilde B\subseteq\Omega\subseteq\widehat B$, where $\widetilde B$ serves the same role as the interior 
tangent ball in the $C^{1,1}$-case and $\widehat{B}$ serves the same role as the complement of the exterior 
tangent ball in the $C^{1,1}$-case. To construct such sets $\widetilde B$
and $\widehat B$, we first place two balls tangent to 
$P_i$ at $p_i$, the closest point to $x$ on $S_i$, for each $i=1, \cdots, k$. We can assign an orientation such that 
we can identify one of these balls as {\it interior} and another as {\it exterior}. As a result, we have 
$k$ interior balls and $k$ exterior balls. Based on how 
$\Omega$ is formed by $S_1, \cdots, S_k$ near $x_0$, we can form 
$\widetilde B$ from the interior balls and $\widehat B$ from the complement of the exterior balls. 
By such an construction, $\widetilde B$ and $\widehat B$ are conformal to infinite cones $\widetilde V$ 
and $\widehat V$. 

In the second step, we compare the solution $u$ in $\Omega$ near $x_0$ with the solution 
$u_{V_{x_0}}$ in the tangent cone $V_{x_0}$. To this end, we first compare $u$ with the solutions 
$\widetilde u$ and $\widehat u$ in $\widetilde B$ and $\widehat B$, respectively, and then compare 
$\widetilde u$ and $\widehat u$ with $u_{V_{x_0}}$. The comparison of $u$ with $\widetilde u$ and $\widehat u$
is based on a simple application of the maximum principle. However, the comparison of 
$\widetilde u$ and $\widehat u$ with $u_{V_{x_0}}$ is delicate and occupies a large portion of the paper. 
Since $\widetilde B$ and $\widehat B$ are conformal to infinite cones $\widetilde V$ 
and $\widehat V$, solutions $\widetilde u$ and $\widehat u$ are related to solutions $\widetilde v$ and 
$\widehat v$ in $\widetilde V$ 
and $\widehat V$ by conformal factors. As a result, we need to compare $u_{V_{x_0}}$ with $\widetilde v$ and 
$\widehat v$, all of which are solutions in cones. Such a comparison is based on 
an anisotropic gradient estimate for solutions in cones. The map $T$ from $\Omega$ to $V_{x_0}$ near $x_0$ 
mentioned above is constructed through the signed distances from $x$ in $\Omega$ to $S_1, \cdots, S_k$ and 
from the corresponding point in $V_{x_0}$ to the faces of $V_{x_0}$. 

The conformal structure of the equation
\eqref{eq-MainEq} plays an essential role. Without such a structure, we will be unable to derive 
the optimal estimate \eqref{main-estimate}. Specifically, if the power $\frac{n+2}{n-2}$ in \eqref{eq-MainEq}
is replaced by other constant $p>1$, we can bound the left-hand side of \eqref{main-estimate} by 
$C|x-x_0|^{\alpha}$, for some constant $\alpha\in (0,1]$, depending on $n$ and $p$. For general $p$, 
there is no conformal structure to utilize. 
The resulted $\alpha$ in fact is much smaller than 1. We will not pursue this issue in the present paper. 


The paper is organized as follows.
In Section \ref{sec-Existence},
we discuss the existence of solutions of \eqref{eq-MainEq}-\eqref{eq-MainBoundary}
in bounded Lipschitz domains and in a certain class of infinite cones.
In Section \ref{sec-BasicEstimates}, we
prove some basic
estimates for asymptotic behaviors near boundary and in particular an 
anisotropic gradient estimate for solutions in cones.
In Section \ref{sec-Domains}, we introduce the class of domains to be discussed in this paper and
analyze their tangent cones.
In Section \ref{sec-DifferentCones}, we compare solutions
in different  cones.
In Section \ref{sec-Singular-1}, we study the
asymptotic
expansions near singular points bounded by $C^{1,1}$-hypersurfaces
and prove  Theorem \ref{main reslut}. 
In Section \ref{Slu-GerSingularDomains}, we discuss the asymptotic expansions 
in more general domains.

\section{Existence of Solutions}\label{sec-Existence}

In this section,
we discuss the existence of solutions of \eqref{eq-MainEq}-\eqref{eq-MainBoundary}
in several classes of domains.

First, we introduce some notations.
Let $x_{0}\in \mathbb R^n$ be a point and $r>0$ be a constant.
Set, for any $x\in B_r(x_0)$,
\begin{equation}\label{eq-SolutionInside}
u_{r, x_0}(x)=\left(\frac{2r}{r^{2}-|x-x_0|^{2}}\right)^{\frac{n-2}{2}}.\end{equation}
Then, $u_{r,x_0}$ is a solution of
\eqref{eq-MainEq}-\eqref{eq-MainBoundary} in $\Omega=B_r(x_0)$.
With $d(x)=r-|x-x_0|$, the distance to the sphere $\partial B_r(x_0)$ from $x\in B_r(x_0)$, we have
$$u_{r,x_{0}}=d^{-\frac{n-2}{2}}\left(1-\frac{d}{2r}\right)^{-\frac{n-2}{2}}.$$
Set, for any  $x\in\mathbb R^n\setminus B_r(x_0)$, 
\begin{equation}\label{eq-SolutionOutside}
v_{r,x_0}(x)=\left(\frac{2r}{|x-x_0|^{2}-r^{2}}\right)^{\frac{n-2}{2}}.\end{equation}
Then, $v_{r,x_0}$ is a solution of
\eqref{eq-MainEq}-\eqref{eq-MainBoundary} in $\Omega=\mathbb R^n\setminus B_r(x_0)$.
With $d(x)=|x-x_0|-r$, the distance to the sphere $\partial B_r(x_0)$ from
$x\in \mathbb R^n\setminus B_r(x_0)$, we have
$$v_{r,x_{0}}=d^{-\frac{n-2}{2}}\left(1+\frac{d}{2r}\right)^{-\frac{n-2}{2}}.$$
These two solutions play an important role in this paper.

Now, we quote a well-known result. 

\begin{theorem}\label{thrm-Existence}
Let $ \Omega $ be a bounded Lipschitz domain in $ \mathbb{R}^{n}$.
Then, \eqref{eq-MainEq} and \eqref{eq-MainBoundary}
admit a unique positive solution $ u \in C^{ \infty}(\Omega)$.
\end{theorem}

Loewner and Nirenberg \cite{Loewner&Nirenberg1974} proved the uniqueness 
for $C^2$-domains. In fact, the uniqueness holds for Lipchitz domains as stated in Theorem 
\ref{thrm-Existence}. 

Next, we state a  basic result which will be needed later.

\begin{lemma}\label{lemma-supersution}
Let $ \Omega $ be a domain in $ \mathbb{R}^{n}$,
and $u$ and $v$ be two nonnegative solutions of  \eqref{eq-MainEq}.
Then, $u+v$ is a
nonnegative supersolution of  \eqref{eq-MainEq}.
\end{lemma}

We omit the proof as it is based on a straightforward calculation. 

\smallskip

In the following, we discuss \eqref{eq-MainEq}-\eqref{eq-MainBoundary}
in infinite cones. Throughout this paper, cones are always solid.

\begin{theorem}\label{existence in cone} For a fixed integer $2\le k\le n$,
let ${V}_{k}$ be an infinite cone in $\mathbb R ^{k}$
such that ${V}_{k}\cap  \mathbb S^{k-1}$ is
a Lipschitz domain in $\mathbb S^{k-1}$.
Then, there exists a solution $u\in C^\infty(V)$ of
\eqref{eq-MainEq}-\eqref{eq-MainBoundary} in $V = {V}_{k}\times \mathbb R^{n-k}$
and $u$ is a function in $x_1, \cdots, x_k$ and satisfies, for any $\lambda>0$ and any 
$x\in V$, 
\begin{equation}\label{eq-scaling} 
u(\lambda x)=\lambda^{-\frac{n-2}{2}}u(x).\end{equation}
\end{theorem}

The scaling property \eqref{eq-scaling} will be used repeatedly in the rest of the paper. 

\begin{proof}
Let $(r,\theta)$ be the polar coordinates in $\mathbb R^k$. Then,
$$\Delta_{\mathbb R^k}=
\frac{\partial^{2}}{\partial r^{2}}+\frac{k-1}{r}\frac{\partial}{\partial
r}+\frac{1}{r^{2}}\Delta_{S^{k-1}},$$
where $\Delta_{\mathbb S^{k-1}}$ is the  Laplace-Beltrami operator on the unit
sphere $  \mathbb S ^{k-1}$. Set
$$u(x)=r^{\alpha} g(\theta),$$
and substitute such a $u$ in \eqref{eq-MainEq} and \eqref{eq-MainBoundary}.
It is easy to see $\alpha=-\frac{n-2}{2}$ and \eqref{eq-MainEq}-\eqref{eq-MainBoundary}
hold if
\begin{align}
\label{eq-MainEq2}\Delta_{ \mathbb S ^{k-1}} g +\left(-\frac{n-2}{2}\right)\left(k-1-\frac{n}{2}\right)g &=
\frac14n(n-2)
g^{\frac{n+2}{n-2}} \quad  \text{in } S_k,\\
\label{eq-MainBoundary2} g &=\infty  \quad\text{on }\partial S_k,
\end{align}
where $S_k={V}_{k}\cap  \mathbb S^{k-1}$.
We will prove that there exists a nonnegative solution $g$ of
\eqref{eq-MainEq2}-\eqref{eq-MainBoundary2}.

For each integer $i\ge 1$, there exists a solution $g_{(i)}\in
C(\overline{S}_{k})\cap C^{\infty}(S_{k}) $ of
\begin{align*}
\Delta_{ \mathbb S ^{k-1}} g_{(i)} +\left(-\frac{n-2}{2}\right)\left(k-1-\frac{n}{2}\right)g_{(i)} &=
\frac14n(n-2)
g_{(i)}^{\frac{n+2}{n-2}} \quad \text{in } S_{k},\\
g_{(i)} &=i\quad\text{on }  \partial S_{k}.
\end{align*}
The proof is based on a standard iteration. 
We now claim  $\Delta_{\mathbb S ^{k-1}}
g_{(i)}\geq 0$ in $S_{k} $, i.e.,
\begin{equation}\label{eq-AlgebraicRelation}\frac{n-2}{2}\left(k-1-\frac{n}{2}\right)g_{(i)} +
\frac14n(n-2)g_{(i)}^{\frac{n+2}{n-2}}
\geq0.\end{equation}
First, \eqref{eq-AlgebraicRelation} obviously holds if $ k-1\geq n/2 $. Next,
we consider the case $ k-1< n/2$. If \eqref{eq-AlgebraicRelation} is violated somewhere,
then $g$ must assume its minimum at some point $\theta_0$ in the set
$$ \left\{\theta\in S^{n}_{k}:\, \frac{n-2}{2}\left(k-1-\frac{n}{2}\right)g_{(i)} +
\frac{n(n-2)}{4}g_{(i)}^{\frac{n+2}{n-2}}< 0 \right\}.$$
On the other hand, we have $
\Delta_{\mathbb S ^{k-1}} g_{(i)} (\theta_{0})\geq 0$, which leads to a contradiction.

By taking a difference, we have
$$\Delta_{\mathbb S ^{k-1}} (g_{(i)}-g_{(i-1)})=c_i(g_{(i)}-g_{(i-1)})\quad \text{in } S_{k},$$
where $c_i$ is a nonnegative function in $S_{k}$. We need \eqref{eq-AlgebraicRelation}
to prove $c_i\ge 0$. 
The maximum principle implies, $g_{(i)}\ge g_{(i-1)}$ for any $i\ge 2$.

For any $\theta \in  S_{k}=V_k\cap\mathbb S^{k-1},$ write $x_\theta=(\theta,0_{\mathbb R^{n-k}})\in\mathbb R^n$.
For an arbitrarily fixed $\theta_{0}\in S_{k},$
take a ball $B_{r_0}(x_{\theta_0})\subset V=V_k\times\mathbb R^{n-k}.$
Then, $u_{(i)}(x)=(\sqrt{x_{1}^{2}+
...+x_{k}^{2}})^{-\frac{n-2}{2}}g_{(i)}(\theta)$ satisfies
\begin{align*}
\Delta u_{(i)} &= \frac14n(n-2)u_{(i)}^{\frac{n+2}{n-2}} \quad \text{in  }
V,\\
u_{(i)} &= (\sqrt{x_{1}^{2}+ ...+x_{k}^{2}})^{-\frac{n-2}{2}}i \quad \text{on }
\partial V.
\end{align*}
Let $u_{r_{0},x_{\theta_0}}$ be the solution of
\eqref{eq-MainEq}-\eqref{eq-MainBoundary} in $B_{r_0}(x_{\theta_0})$,
given by \eqref{eq-SolutionInside}.
By  the maximum principle, we have, for any $x\in B_{r_0}(x_{\theta_0}),$
$$u_{(i)}(x)\leq u_{r_{0},x_{\theta_0}}(x).$$
Then, for any $\theta \in  S_{k}$ with $x_\theta\in
B_{r_0}(x_{0})$,
$$g_{(i)}(\theta)=u_{(i)}(x_\theta)\leq
u_{r_{0},x_{\theta_0}}(x_\theta).$$ Therefore, there exists a  $ g \in C^{\infty}(S_{k})$ such
that $g_{(i)}\rightarrow $ $g$ in $  C^{m}_{loc}(S_{k}) $, for any positive integer  $m$.
 Since $g_{(i)}$ equals $i$ on $\partial S_{k}$, $g$ is a solution of
\eqref{eq-MainEq2}-\eqref{eq-MainBoundary2}.
\end{proof}

If the cone $ V$ as in Lemma \ref{existence in cone} is Lipschitz  near the origin, then the nonnegative solution $u$
of \eqref{eq-MainEq}-\eqref{eq-MainBoundary} for $V$ is unique.
To verify this, let $\widetilde{u}$ be another nonnegative solution of
\eqref{eq-MainEq}-\eqref{eq-MainBoundary}. By Lemma \ref{lemma-supersution} and the
maximum principle, we have, for any $\epsilon $ small  and $ r\gg |x| $ large,
\begin{align*}
\widetilde{u}(x)& \leq u (x-\epsilon {e} ) + u_{r,0}(x), \\
\widetilde{u}(x)& \geq u (x+\epsilon {e} ) - u_{r,0}(x),
\end{align*}
where ${e}$ is some
unit vector in $V$ and $u_{0,r}$ is the solution of \eqref{eq-MainEq}-\eqref{eq-MainBoundary} in $B_r$.
Letting $\epsilon\rightarrow 0$ and $ r\rightarrow \infty $, we obtain $ \widetilde{u}(x) =
u(x) $, which implies the uniqueness.

\section{Basic Estimates}\label{sec-BasicEstimates}

In this section, we prove several basic estimates concerning asymptotic behaviors
of solutions near boundary.
First, we study the
asymptotic behavior near $C^{1,\alpha}$-portions of  $\partial\Omega$.

\begin{theorem}\label{thrm-C-1,alpha-expansion}
Let $\Omega$ be a bounded domain in $\mathbb R^n$ and $\partial\Omega$ be
$C^{1,\alpha}$ near $x_0\in\partial\Omega$ for some $\alpha\in (0,1]$. Suppose
$u\in C^\infty(\Omega)$ is a solution of \eqref{eq-MainEq}-\eqref{eq-MainBoundary}. Then,
$$| d^{\frac{n-2}{2}}u-1|\leq Cd^{\alpha}\quad\text{in }\Omega\cap B_r(x_0),$$
where $d$ is the distance to $\partial\Omega$, and $r$ and $C$
are positive constants depending only on $n$, $\alpha$ and the geometry of  $\Omega$.
\end{theorem}

\begin{proof}
We take $R>0$ sufficiently small such that $\partial\Omega\cap B_{R}(x_0)$ is
$C^{1,\alpha}$.
We fix an $x\in\Omega\cap B_{R/4}(x_0)$ and take $p\in \partial\Omega$, also
near $x_0$, such that
$d(x)=|x-p|$.
Then,  $p\in  \partial\Omega\cap B_{R/2}(x_0)$.
By a translation and rotation, we assume $p=0$ and the $x_n$-axis is the interior normal
to $\partial\Omega$
at 0. Then, $x$ is on the positive $x_n$-axis,
with $d=d(x)=|x|$, and $P=\{x_n=0\}$ is the tangent plane of $\partial\Omega$ at 0.
Moreover, a portion of $\partial \Omega$ near 0 can be expressed by a
$C^{1,\alpha}$-function $\varphi$ in $B'_{R}\subset\mathbb R^{n-1}$, with $\varphi(0)=0$ and
\begin{equation}\label{eq-boundaryC1alpha}
|\varphi(x')|\le M|x'|^{1+\alpha}\quad\text{for any } x'\in B_R'.\end{equation}
Here,  $M$ is a  positive constant chosen to be uniform, independent of $x$.

We first consider the case $\alpha=1$. For any $r>0$,  the lower semi-sphere of
$$x_1^2+...+x_{n-1}^2+(x_n-r)^2=r^2$$ satisfies
$x_n\ge |x'|^2/(2r)$.
By fixing a constant $r$ sufficiently small,
\eqref{eq-boundaryC1alpha} implies
$$B_r(re_n)\subset\Omega\text{ and }B_r(-re_n)\cap \Omega=\emptyset.$$
Let $u_{r, re_n}$ and $v_{r, -re_n}$ be the solutions of
\eqref{eq-MainEq}-\eqref{eq-MainBoundary} in $B_r(re_n)$ and $\mathbb R^n\setminus
B_r(-re_n)$, given by \eqref{eq-SolutionInside}
and \eqref {eq-SolutionOutside},
respectively. Then, by the maximum principle, we have
$$v_{r,-re_n}\le u\le u_{r, re_n}\quad\text{in }B_r(re_n).$$
For the $x$ above in the positive $x_n$-axis with $|x|=d<r$, we obtain
$$d^{-\frac{n-2}{2}}\left(1+\frac{d}{2r}\right)^{-\frac{n-2}{2}}\le u \le
d^{-\frac{n-2}{2}}\left(1-\frac{d}{2r}\right)^{-\frac{n-2}{2}}.$$
This implies the desired result for $\alpha=1$.

Next, we consider $\alpha\in (0,1)$. Recall that $x$ is in the positive $x_n$-axis and
$|x|=d$.
We first note
\begin{equation}\label{eq-boundaryC1alpha2}
|x'|^{1+\alpha}\le d^{1+\alpha}+\frac{1}{d^{1-\alpha}}|x'|^{2} \quad\text{for any
}x'\in\mathbb R^{n-1}.\end{equation}
This follows from the H\"older inequality, or more easily, by considering $|x'|\le d$ and
$|x'|\ge d$
separately. Let $r=d^{1-\alpha}/(2M)$ and $q$ be the point on the positive $x_n$-axis such
that
$|q|=Md^{1+\alpha}+r$. By taking $d$ sufficiently small, \eqref{eq-boundaryC1alpha}
and \eqref{eq-boundaryC1alpha2} imply
$$B_r(q)\subset\Omega\text{ and }B_r(-q)\cap \Omega=\emptyset.$$
Let $u_{r, q}$ and $v_{r, -q}$ be the solutions of
\eqref{eq-MainEq}-\eqref{eq-MainBoundary} in $B_r(q)$ and $\mathbb R^n\setminus B_r(-q)$,
given by \eqref{eq-SolutionInside}
and \eqref {eq-SolutionOutside},
respectively. Then, by the maximum principle, we have
$$v_{r,-q}\le u\le u_{r, q}\quad\text{in }B_r(q).$$
For the $x$ above,
dist$(x, \partial B_r(q))=d-Md^{1+\alpha}$ and
dist$(x, \partial B_r(-q))=d+Md^{1+\alpha}$. Evaluating at such an $x$, we obtain
\begin{align*}&(d+Md^{1+\alpha})^{-\frac{n-2}{2}}
\left(1+\frac{M}{d^{1-\alpha}}(d+Md^{1+\alpha})\right)^{-\frac{n-2}{2}}\\
&\quad \le u\le
(d-Md^{1+\alpha})^{-\frac{n-2}{2}}
\left(1+\frac{M}{d^{1-\alpha}}(d-Md^{1+\alpha})\right)^{-\frac{n-2}{2}}.\end{align*}
This implies the desired result for $\alpha\in (0,1)$. \end{proof}

If $\partial\Omega$ is $C^{1,\alpha}$ near $x_0$ for some $\alpha\in (0,1]$, then
the tangent cone $V_{x_0}$ of $\Omega$ at $x_0$ is a half space and
$v(x)=d(x)^{-\frac{n-2}{2}}$ is  the solution of Loewner-Nirenberg
problem in $V_{x_0}$, where $d$ is the distance to $\partial V_{x_0}$.

\smallskip

Next, we prove a preliminary result for domains with singularity. We note that a finite
circular cone
is determined by its vertex, its axis, its height and its opening angle. The height and the opening
angle are often referred to as the size of the cone. Here, cones are solid.
We point out that we do not assume the boundedness of the domains in the next result.

\begin{lemma}\label{lemma-ExteriorCone}
Let $ \Omega $ be a domain satisfying the uniform exterior cone condition in $\mathbb{R}^{n}$.
Suppose $ u \in C^{ \infty}(\Omega)$ is a solution of \eqref{eq-MainEq}-\eqref{eq-MainBoundary}.
Then there exists a constant $\delta >0$,  depending only on the size of exterior cones,
such that, for any $x\in\Omega$ with $d(x)<\delta$,
$$C^{-1}
\leq d(x)^{\frac{n-2}{2}}u(x) \leq 2^{\frac{n-2}{2}},$$ where $d(x) $ is the distance from
$x$ to $\partial\Omega$, and $C$ is a constant  depending only on $n$ and the size of exterior
cones.
\end{lemma}

\begin{proof}
We take an arbitrarily fixed $ x \in \Omega$. Then,  $ B_{d(x)}(x) \subseteq \Omega$.
Let $u_{d(x),x}$ be the solution of
\eqref{eq-MainEq}-\eqref{eq-MainBoundary} in $B_{d(x)}(x) $,
given by \eqref{eq-SolutionInside}.
By the maximum
principle, we have $$u(x) \leq u_{d(x),x} (x)=
d(x)^{-\frac{n-2}{2}}\left(1-\frac{d(x)}{2d(x)}
\right)^{-\frac{n-2}{2}}=2^{\frac{n-2}{2}}d(x)^{-\frac{n-2}{2}}.$$
Next, we assume $d(x)=|x-p|,$ for some $p\in\partial\Omega$.
There exists a finite circular cone $V_p$, with the vertex $p$, the axis
${e}_{p}$,
the height $h$,  and the opening angle $2\theta$, such that
$V_{p}\subseteq \Omega^{c}$. Here, we can assume $h$ and $\theta$ are
constants independent of  the choice of
$p \in \partial \Omega.$ We further assume
\begin{equation}\label{eq-Requirement-d}d(x)<h(1+\frac{1}{\sin\theta})^{-1}.\end{equation}
Set $\widetilde{p}= p+\frac{1}{\sin\theta}d(x)
{e}_{p}.$ Then, $B_{d(x)}(\widetilde{p})\subseteq\Omega^{c}.$
Let $v_{d(x),\widetilde{p}}$ be the solution of
\eqref{eq-MainEq}-\eqref{eq-MainBoundary} in  $\mathbb R^n\setminus B_{d(x)}(\widetilde{p})$,
given by  \eqref {eq-SolutionOutside}.
By the maximum principle, we have
$$\aligned u(x) &\geq v_{d(x),\widetilde{p}}(x)\geq
\left(d(x)+\frac{1}{\sin\theta}d(x)\right)^{-\frac{n-2}{2}}
\left(1+\frac{d(x)+\frac{1}{\sin\theta}d(x)}{2d(x)}\right)^{-\frac{n-2}{2}}\\
&\geq Cd(x)^{-\frac{n-2}{2}}.\endaligned$$
We conclude the desired estimate. \end{proof}

Lemma \ref{lemma-ExteriorCone} demonstrates that, for any $x\in \Omega$, the values of solutions
in $B_{d(x)/2}(x)$ are comparable.

\begin{remark}\label{rmrk-delta} The requirement on $d(x)<\delta$ in Lemma \ref{lemma-ExteriorCone}
is due to \eqref{eq-Requirement-d}, where $h$ is the height of the exterior cone.
If at each point $p\in \partial\Omega$, there exists an infinite exterior cone with a fixed angle $\theta$, then
Lemma \ref{lemma-ExteriorCone} holds for all $x\in \Omega$. A similar remark also holds
for Lemma \ref{lemma-GradientEstimates} below.
\end{remark}

We next derive estimates of derivatives.

\begin{lemma}\label{lemma-GradientEstimates}
Let $ \Omega $ be a domain satisfying the uniform exterior cone condition in $\mathbb{R}^{n}$.
Suppose $ u \in C^{ \infty}(\Omega)$ is a solution of \eqref{eq-MainEq}-\eqref{eq-MainBoundary}.
Then there exists a constant $\delta >0$,  depending only on the size of exterior cones,
such that, for any $x\in\Omega$ with $d(x)<\delta$,
\begin{equation}\label{eq-Derivative}
d(x)|Du(x)|+d^2(x)|D^2u(x)|\le Cu(x),\end{equation} where $d(x) $ is the distance from
$x$ to $\partial\Omega$, and $C$ is a constant  depending only on $n$ and the size of exterior
cones.
\end{lemma}

\begin{proof} We take an arbitrary $\beta\in (0,1)$. By the standard
interior estimates, we have, for any $x\in \Omega$ with
$d=\operatorname{dist}(x,\partial \Omega)$,
$$\aligned&d^{\beta}[u]_{C^{\beta}(B_{{d}/{8}}(x))}
+d|u|_{L^{\infty}(B_{{d}/{8}}(x))}+d^{1+\beta}[Du]_{C^{\beta}(B_{{d}/{8}}(x))}\\
&\qquad\leq
C\{|u|_{L^{\infty}(B_{{d}/{4}}(x))}
+d^{2}|u^{\frac{n+2}{n-2}}|_{L^{\infty}(B_{{d}/{4}}(x))}\},\endaligned$$
and
$$\aligned &d^{2}|D^{2}u|_{L^{\infty}(B_{{d}/{16}}(x))}+d^{2+\beta}[D^2u]_{C^{\beta}(B_{{d}/{8}}(x))}\\
&\qquad\leq
C\{|u|_{L^{\infty}(B_{{d}/{8}}(x))}
+d^{2}|u^{\frac{n+2}{n-2}}|_{L^{\infty}(B_{{d}/{8}}(x))}  +
d^{2+\beta}[u^{\frac{n+2}{n-2}}]_{C^{\beta}(B_{{d}/{8}}(x))}\},\endaligned$$
where $C$ is a positive constant depending only on $n$ and $\beta$. Then,
Lemma \ref{lemma-ExteriorCone} implies the desired result.
\end{proof}

In \eqref{eq-Derivative}, the gradients at points are estimated in terms of their distances to
the boundary. This estimate is not sufficient in many applications.
In the next result, we estimate the directional derivatives in cones
along certain directions in terms of the distance to the corresponding faces forming the boundary of the cones.
Such an anisotropic gradient estimate plays a fundamental role in this paper.

\begin{lemma}\label{lemma-AnisotropicEstimate}
Let $P_1, \cdots, P_n$ be $n$ hyperplanes in $\mathbb R^n$
passing the origin with linearly independent unit normal vectors $\nu_1, \cdots, \nu_n$,
and $V\subset\mathbb R^n$ be an infinite cone, with its vertex at the origin
and a Lipschitz boundary,
such that
$$\partial V=\bigcup_{i=1}^nF_i,$$
where $F_i$ is a subset of $P_i$ with a nonempty interior.
Suppose $u$ is the
unique nonnegative solution of \eqref{eq-MainEq}-\eqref{eq-MainBoundary} for
$\Omega=V$.
Then, for each $i=1, \cdots, n$ and any $x\in V$,
\begin{equation}\label{eq-anisotropic}
\text{dist}(x, F_i)\left|\partial_{\mu_i} u(x)\right| \leq Cu(x),
\end{equation}
where
$\mu_{_i}$ is a unit vector along $\bigcap_{j \neq i} F_j$,
and $C$ is a positive constant depending only on $n$ and
$\|(\nu_1, \cdots, \nu_n)^{-1}\|$.
Moreover, for any $x, x^*\in V$, if, for any $i=1, \cdots, n$, 
\begin{equation}\label{eq-AnisotripicCondition}
|\langle x-x^*, \mu_i\rangle|\le \tau|\langle x, \mu_i\rangle|,\end{equation}
for some constant $\tau>0$, then 
\begin{equation}\label{eq-AnisotropicDifference} 
|u(x)-u(x^*)|\le C\tau u(x).\end{equation}\end{lemma}

Here and hereafter, $\|\cdot\|$ is the norm of matrices, considered as transforms in the Euclidean spaces.
We always treat $\nu_i$ as a column vector. Then, $(\nu_1, \cdots, \nu_n)$ is an invertible
$n\times n$ matrix. In the statement of Lemma \ref{lemma-AnisotropicEstimate}, each
$F_i$ is a face of $V$ and each $\bigcap_{j \neq i} F_j$ is an edge transversal to $F_i$. Hence
in \eqref{eq-anisotropic}, directional derivatives along edges are estimated in terms of the distances
to the corresponding faces.

\begin{proof} Without loss of generality, we assume 
$\langle\nu_i, \mu_i\rangle>0$, for any $i=1, \cdots, n$. 
Let $e_i$ be the unit vector along the $x_i$-axis, for each $i=1, \cdots, n$.
Consider the linear transform $E$ given by 
$$E=(\mu_{1 },\cdots,\mu_{n } )^{-1}.$$
Then, $E$ transforms $\mu_i$ to $e_{i }$ and $P_i$ to the hyperplane $\{ x_i =0\}$.
Set $\widetilde V=EV$ and, for $\widetilde x \in \widetilde V$,
$$\widetilde{u}(\widetilde x) =u(E^{-1}\widetilde x).$$
Under the transform $\widetilde x=Ex$ with $x \in V$, we have
\begin{equation}\label{eq-TransformedEq}a_{ij} \widetilde{u}_{\widetilde x_i\widetilde x_j} =
\frac14n(n-2) \widetilde{u}^{\frac{n+2}{n-2}}\quad\text{in }\widetilde V,\end{equation}
where $(a_{ij}) = EE^T$. We also  have, for $i=1, \cdots, n$,
\begin{equation}\label{eq-RelationSignedDistance}
d_i =c_i|\widetilde x_i|,\end{equation}
where $d_i$ is the distance to $P_i$ and $c_i$ is a constant satisfying
$$\frac{1}{\| E\|} \leq c_i \leq \|E^{-1}\|.$$
We first note $\|E^{-1}\|\le \sqrt{n}$. Next, we claim
\begin{equation}\label{eq-Estimate-E}\| E\| \leq \sqrt{n} \|(\nu_1,\cdots,\nu_n)^{-1}\|.\end{equation}
To prove this, we consider the $n\times n$ matrix $N$ given by
$$N=(\nu_1, \cdots,\nu_n).$$
Then,
$\langle \nu_l, \mu_{j}\rangle=0$ for $l\neq j.$
As a consequence, for any $j=1, \cdots, n$,
\begin{equation}\label{eq-Relation}N^{T}\mu_{j} =\langle \nu_j, \mu_{j}\rangle e_j,\end{equation}
and hence
$\| N^{-1}\| \geq{\langle \nu_{1,j},
\mu_{1,j}\rangle}^{-1}$. By writing \eqref{eq-Relation} in its matrix form
$$N^TE^{-1}=\operatorname{diag}(\langle \nu_1, \mu_{1}\rangle, \cdots, \langle \nu_n, \mu_{n}\rangle),$$
we obtain
\begin{align*}
\|E \|  \leq \|  N^{T}\|\cdot   \|N^{-1}\|
\leq \sqrt{n}\| N^{-1}\|.
\end{align*}
This is \eqref{eq-Estimate-E}.
Now, $\widetilde V$ is bounded by faces $\widetilde F_1, \cdots, \widetilde F_n$, with
each $\widetilde F_i=EF_i$ on a hyperplane.
Without loss of generality, we assume, for each $i=1, \cdots, n$,
$$\widetilde F_i\subset \{\widetilde x_i=0\}.$$
Set, for any $\widetilde x\in \widetilde V$,
$$\widetilde{d}_i =\text{dist}(\widetilde x, \widetilde F_i).$$
We will prove
\begin{equation}\label{eq-gradient-estimate-cone}
|\widetilde{u}_{\widetilde x_i}| \leq C\frac{\widetilde{u}}{\widetilde{d}_i}.
\end{equation}
Without loss of generality, we assume, for a fixed $\widetilde x\in \widetilde V$,
$$\widetilde{d}_1 \leq \widetilde{d}_2 \leq \cdots\leq \widetilde{d}_n.$$
We note $\widetilde{d}_1 =\text{dist}(\widetilde x, \partial \widetilde V).$

{\it Case 1.} First, we prove \eqref{eq-gradient-estimate-cone} for $i=1$.
By Lemma \ref{lemma-ExteriorCone}
and Lemma \ref{lemma-GradientEstimates}, we have, for $i=1,...,n$,
\begin{equation}\label{eq-case1}
|\widetilde u_{\widetilde x_i}|\leq  C\widetilde{d}_1^{-\frac{n-2}{2}-1}   \leq  C\frac{\widetilde{u}}{\widetilde{d}_1}.\end{equation}
In particular, we have \eqref{eq-gradient-estimate-cone} for $i=1$.

{\it Case 2.} Next, we prove \eqref{eq-gradient-estimate-cone} for $i=2$.

{\it Case 2.1.} We assume $\widetilde{d}_2 \leq  8n\|E\| \widetilde{d}_1$.
Then, by \eqref{eq-case1},
$$|\widetilde{u}_{\widetilde x_2}|\leq  C\frac{\widetilde{u}}{\widetilde{d}_1}   
\leq  C\frac{\widetilde{u}}{\widetilde{d}_2}.$$

{\it Case 2.2.} We assume $\widetilde{d}_2 > 8n\|E\| \widetilde{d}_1$.
Set $r ={\widetilde{d}_2}/{8}$.
Then,
$$B_{8r}(\widetilde x)\bigcap \partial \widetilde V \subseteq \{\widetilde x_1= 0\}.$$
Let $\widetilde p$ be the point on $\{\widetilde x_1= 0\}$ with the smallest
distance to $\widetilde x$. Then, $B_{4r}(\widetilde p)\bigcap \partial \widetilde V \subseteq \{\widetilde x_1= 0\}$, 
and $B_{4r}(\widetilde p)\bigcap \partial \widetilde V\cap\{\widetilde x_i=0\}=\emptyset$ for $i=2,\cdots, n$.

Set
$$\overline{u}(\overline x) = r^{\frac{n-2}{2}}{\widetilde u}(\widetilde p+r\overline x).$$
By Theorem \ref{thrm-C-1,alpha-expansion},
we have
\begin{equation}\label{eq-Estimate-overline-w}|(c_1 \overline x_1)^{\frac{n-2}{2}}\overline{u}
-1| \leq C \overline{d},\end{equation}
where $c_1$ is as in \eqref{eq-RelationSignedDistance}.
Although Theorem \ref{thrm-C-1,alpha-expansion} is formulated for bounded domains, the proof
only requires the existence of the interior and exterior tangent balls.
We also point out that Theorem \ref{thrm-C-1,alpha-expansion} holds for solutions of
\eqref{eq-MainEq} and that $\widetilde u$ satisfies
\eqref{eq-TransformedEq}. Therefore, there is an extra factor $c_1$ in \eqref{eq-Estimate-overline-w}.

Set
$$\overline w=(c_1 \overline x_1)^{\frac{n-2}{2}}\overline{u}-1.$$
We denote by $\overline V$ the image of $\widetilde V$ under the transform $(\cdot-\widetilde p)/r$. Then,
$ \overline w$ satisfies
\begin{equation*}
a_{ij}\overline w_{\overline x_i\overline x_j}+ 2a_{ij}
\frac{[(c_1 \overline x_1)^{-\frac{n-2}{2}}]_{\overline x_i}}{(c_1 \overline x_1)^{-\frac{n-2}{2}}} \overline w_{\overline x_j}=
\frac14n(n-2) ( c_1 \overline x_1)^{-2} [(1+\overline w)^{\frac{n+2}{n-2}}-1-\overline w],
\end{equation*}
in $B_4 \cap \overline V$.
For any fixed $\overline{x} \in B_{2} \cap \overline V$, we consider the above equation in 
$B_{d_{\overline{x}}/2}(\overline x)$. 
First, \eqref{eq-Estimate-overline-w} implies 
$$|\overline w|\le Cd_{\overline x}\quad\text{in }B_{d_{\overline{x}}/2}(\overline x).$$
Next, we have 
$$\left|d_{\overline{x}}a_{ij}
\frac{[(c_1 \overline x_1)^{-\frac{n-2}{2}}]_{\overline x_i}}{(c_1\overline x_1)^{-\frac{n-2}{2}}}\right|
\le C\quad\text{in }B_{d_{\overline{x}}/2}(\overline x),$$ 
and 
$$\left|d^2_{\overline{x}}(c_1 \overline x_1)^{-2}
[(1+\overline{w})^{\frac{n+2}{n-2}}-1-\overline{w}]\right|
\leq Cd_{\overline x}\quad\text{in }B_{d_{\overline{x}}/2}(\overline x).$$
By applying the scaled interior estimate in $B_{d_{\overline{x}}/2}(\overline x)$, we get, for $i=1,\cdots, n$,
$$|\overline w_{\overline x_i} (\overline{x}) | \leq C,$$
and then, for  $i=2,\cdots,n$,
$$|\overline{u}_{\overline x_i}| =|[(1+\overline w)(c\overline x_1)^{-\frac{n-2}{2}}]_{\overline x_i}|
=|\overline w_{\overline x_i}(c\overline x_1)^{-\frac{n-2}{2}}|\leq C\overline{u},$$
where we used Lemma \ref{lemma-ExteriorCone}. Therefore, for $i=2,...,n$,
\begin{equation}\label{eq-case2}|\widetilde{u}_{\widetilde x_i}|
=|[ r^{-\frac{n-2}{2}}\overline{u}(\frac{\widetilde x-\widetilde p}{r})]_{\widetilde x_i}|
\leq C\frac{\widetilde{u}}{r}\leq  C\frac{{u}}{\widetilde{d}_2}.\end{equation}
Hence, we have  \eqref{eq-gradient-estimate-cone} for $i=2$.

{\it Case 3.} Next, we prove  \eqref{eq-gradient-estimate-cone} for $i=3$.

{\it Case 3.1.} We assume 
$\widetilde{d}_3 \leq  8n\|E\| \widetilde{d}_2$, 
$\widetilde{d}_2 \leq  8n\|E\| \widetilde{d}_1$.
Then, by \eqref{eq-case1},
$$|\widetilde{u}_{\widetilde x_3}|\leq  C\frac{\widetilde{u}}{\widetilde{d}_1}  
\leq  C\frac{\widetilde{u}}{\widetilde{d}_3}.
$$

{\it Case 3.2.} We assume $\widetilde{d}_3 \leq  8n\|E\| \widetilde{d}_2$,
$\widetilde{d}_2 > 8n\|E\| \widetilde{d}_1$. Then, by \eqref{eq-case2},  we have
$$|\widetilde{u}_{\widetilde x_3}|
\leq  C\frac{{u}}{\widetilde{d}_2} \leq C\frac{{u}}{\widetilde{d}_3}.$$

{\it Case 3.3.} We assume $\widetilde{d}_3> 8n \|E\|\widetilde{d}_2$.
We first consider the case $n=3$. For the given $\widetilde x=(\widetilde x_1, \widetilde x_2, \widetilde x_3)$
and any $t$ sufficiently small, consider
$\widetilde x'=(\widetilde x_1, \widetilde x_2, (1+t)\widetilde x_3)$
and $\widetilde x''=(1+t)\widetilde x$. The mean-value theorem implies
$$|\widetilde u(\widetilde x')-\widetilde u(\widetilde x'')|\le 
|\widetilde u_{\widetilde x_1}(\widetilde \xi_1)t\widetilde x_1|
+|\widetilde u_{\widetilde x_2}(\widetilde \xi_2)t\widetilde x_2|,$$
where $\widetilde \xi_1, \widetilde \xi_2$ are points on the line segment between $\widetilde x'$
and $\widetilde x''$. By what we proved in Case 1 and Case 2, we have
$$|\widetilde u(\widetilde x')-\widetilde u(\widetilde x'')|\le
C\left(\frac{\widetilde u(\widetilde \xi_1)}{\widetilde d_1(\widetilde \xi_1)}|t\widetilde x_1|
+\frac{\widetilde u(\widetilde \xi_2)}{\widetilde d_2(\widetilde \xi_2)}|t\widetilde x_2|\right)
\leq C|t|\widetilde{u}(\widetilde x).$$
Next, by the scaling property \eqref{eq-scaling} of solutions in cones, we get
$$|\widetilde u(\widetilde x)-\widetilde u(\widetilde x'')|
= | (1+t)^{-\frac{n-2}{2}} -1|\widetilde{u}(\widetilde x)\leq C|t|\widetilde{u}(\widetilde x),$$
and hence
$$|\widetilde u(\widetilde x)-\widetilde u(\widetilde x')|
\leq C|t|\widetilde{u}(\widetilde x).$$
Dividing by $|t\widetilde x_3|$ and letting $t\to 0$, we obtain
$$|\widetilde{u}_{\widetilde x_3}(\widetilde x)|\leq  C\frac{\widetilde{u}(\widetilde x)}{|\widetilde{x}_3|}.$$
Note
$$\widetilde d_3^2\le \widetilde x_1^2+\widetilde x_2^2+\widetilde x_3^2
\le \widetilde d_1^2+\widetilde d_2^2+\widetilde x_3^2.$$
By the assumptions on $\widetilde d_1, \widetilde d_2, \widetilde d_3$, 
we obtain $\widetilde d_3\le 2|\widetilde x_3|$
and hence
$$|\widetilde{u}_{\widetilde x_3}(\widetilde x)|\leq  C\frac{\widetilde{u}(\widetilde x)}{\widetilde{d}_3}.$$

Next, we consider the case $n\ge 4$.
Set $r ={\widetilde{d}_3}/{8}$.
Then,
$$B_{8r}(\widetilde x)\cap \partial \widetilde V \subseteq \{\widetilde x_1= 0\}\cup\{\widetilde x_2= 0\}.$$
Let $\widetilde V_{12}$ be the infinite cone bounded by $\{\widetilde x_1= 0\}$ and
$\{\widetilde x_2= 0\}$ such that $B_{8r}(\widetilde x)\cap \widetilde V=B_{8r}(\widetilde x)\cap \widetilde V_{12}$.
Set $\widetilde p$ to be the point on $\partial \widetilde V_{12}$ with the smallest
distance to $\widetilde x$. Then, $B_{4r}(\widetilde p)\cap \partial \widetilde V \subseteq \partial \widetilde V_{12}$.
Let $\widetilde v$ be the nonnegative solution of \eqref{eq-TransformedEq} and \eqref{eq-MainBoundary}
for $\Omega = \widetilde V_{12}$,
which is a function of only two variables $\widetilde x_1$ and $\widetilde x_2$.

Set $$\overline{u}(\overline x) = r^{\frac{n-2}{2}}{\widetilde u}(\widetilde p+r\overline x), $$
and
$$\overline{v}(\overline x) = r^{\frac{n-2}{2}}\widetilde v(\widetilde p+r\overline x).$$
We denote by $\overline V$ and $\overline V_{12}$ the images of $\widetilde V$ and $\widetilde V_{12}$ under
the transform $(\cdot-\widetilde p)/r$, respectively.
Set
$$\overline w=\frac{\overline{u}}{\overline{v}}-1.$$ Then, $ \overline w$ satisfies
\begin{equation*}
a_{ij}\overline w_{\overline x_i\overline x_j}
+ 2a_{ij}\frac{\overline{v}_{\overline x_i}}{\overline{v}} \overline w_{\overline x_j}=
\frac14n(n-2)\overline{ v }^{\frac{4}{n-2}}[(1+\overline w)^{\frac{n+2}{n-2}}-1-\overline w]
\quad\text{in }B_4\cap \overline V_{12}.
\end{equation*}
Next, we claim
\begin{equation}\label{eq-Estimate_w}
|\overline w| \leq C \overline{d}^{\frac{n-2}{2}}\quad\text{in }B_2\cap \overline V.\end{equation}
For any fixed $\overline{x} \in B_2 \cap \overline V$, we note
$\overline V\cap B_{1/2}(\overline x)\subset \overline V_{12}$ and
$\overline V_{12}\cap B_{1/2}(\overline x)\subset \overline V$.
Let $u_{\overline x, 1/2}$ be the solution of \eqref{eq-MainEq}-\eqref{eq-MainBoundary} in $B_{1/2}(\overline x)$,
given by \eqref{eq-SolutionInside}.
Lemma \ref{lemma-supersution} implies
\begin{align*}
\overline u\le \overline v+u_{\overline x, 1/2},\quad \overline u\le \overline v+u_{\overline x, 1/2}
\quad\text{in }B_{1/2}(\overline x).
\end{align*}
Note $u_{\overline x, 1/2}(\overline x)\le  C$. By Lemma \ref{lemma-ExteriorCone}, we have
\begin{align*}
\overline u(\overline x)&\le \overline v(\overline x)+C\le \overline v(\overline x)(1+C\overline d^{\frac{n-2}{2}}),\\
\overline v(\overline x)&\le \overline u(\overline x)+C\le \overline u(\overline x)(1+C\overline d^{\frac{n-2}{2}}).
\end{align*}
This finishes the proof of \eqref{eq-Estimate_w}. By $n\ge 4$, we get
\begin{equation}\label{eq-Estimate-overline-w2}|\overline w| 
\leq C \overline{d}\quad\text{in }B_2\cap \overline V.\end{equation}
For any fixed $\overline{x} \in B_2 \cap \overline V$, we have 
$$|\overline w|\le Cd_{\overline x}\quad\text{in }B_{d_{\overline x}/2}(\overline{x}).$$
Moreover, 
$$\left|d_{\overline{x}}a_{ij}\frac{\overline{v}_{x_j}}{\overline{v}}\right|
\le C\quad\text{in }B_{d_{\overline x}/2}(\overline{x}),$$
and 
$$
d^2_{\overline{x}}\overline{ v }^{\frac{4}{n-2}}
|(1+\overline{w})^{\frac{n+2}{n-2}}-1-\overline{w}|\leq C d_{\overline{x}}
\quad\text{in }B_{d_{\overline x}/2}(\overline{x}).$$
By applying the scaled interior estimate in $B_{d_{\overline x}/2}(\overline{x})$, we have, 
for $i=1,\cdots, n$,
$$|\overline w_{\overline x_i} (\overline x) | \leq C,$$
and then, for $i=3,\cdots, n$,
$$|\overline{u}_{\overline x_i}| =|[(1+\overline w)\overline{v}]_{\overline x_i}|
=|\overline w_{\overline x_i}\overline{v}|\leq C\overline{u},$$
where we used Lemma \ref{lemma-ExteriorCone} and
the fact that $\overline v$ is a function of $\overline x_1$ and $\overline x_2$. Therefore, for $i=3,\cdots, n$,
$$|\widetilde{u}_{\widetilde x_i}| =|[ r^{-\frac{n-2}{2}}\overline{u}(\frac{\widetilde x-\widetilde p}{r})]_{\widetilde x_i}|
\leq C\frac{\widetilde{u}}{r}\leq  C\frac{\widetilde{u}}{\widetilde{d}_3}.$$

Similarly, we can prove \eqref{eq-gradient-estimate-cone} for general $i$. 
This is \eqref{eq-anisotropic} in $\widetilde V$. 

Next, we prove \eqref{eq-AnisotropicDifference}. In $\widetilde V$, 
\eqref{eq-AnisotripicCondition} reduces to 
$$|\widetilde{x}_i-\widetilde{x}^*_i|\le \tau|\widetilde{x}_i|.$$
By \eqref{eq-gradient-estimate-cone} and the fact $|\widetilde{x}_i|\le \widetilde{d}_i$, we have  
$$
|\widetilde{u}(\widetilde{x})-\widetilde{u}(\widetilde{x}^*)|\le C\tau \widetilde{u}(\widetilde{x}).$$
This is \eqref{eq-AnisotropicDifference} in $\widetilde V$. \end{proof}

Lemma \ref{lemma-AnisotropicEstimate}, appropriately modified,
holds if the boundary $\partial V$ is formed from $k$ linearly
independent hyperplanes, for some $k<n$. In this case, we can assume $V=V_k\times\mathbb R^{n-k}$, where
$V_k$ is an infinite cone in $\mathbb R^k$, and then by Lemma \ref{existence in cone},
the solution $u$ is a function of $(x_1, \cdots, x_k)\in \mathbb R^k$.

\section{Domains and Their Tangent Cones}\label{sec-Domains}

In this section, we discuss bounded Lipchitz domains whose boundaries consist of finitely many 
$C^{1,1}$-hypersurfaces locally and focus on the relation between these domains and their tangent cones. 
To do this, we will place these domains in a good position and express boundaries of the Lipchitz domains 
and boundaries of the tangent cones by functions satisfying the same algebraic relation. At the first glance, 
this is a tedious way to describe such a relation and does not seem necessary. As mentioned in the introduction, 
an important step in the derivation of the asymptotic expansion in the proof of Theorem \ref{main reslut}
is to construct two sets, one inside the domains and another containing the domains. The algebraic 
relation governing the relation between the domains and their tangent cones will provide an easy 
description to construct these two sets. As we will see, the same algebraic relation determines four sets, 
the domains bounded by $k$ $C^{1,1}$-hypersurfaces, their tangent cones bounded by 
$k$ hyperplanes, the sets inside the domains bounded by $k$ spheres, and the sets containing the domains 
also bounded by $k$ spheres.

We start our discussion with infinite cones,  which
serve as tangent cones. We emphasize  that
all cones in this paper are solid.

Any finitely many hyperplanes passing the origin divide $\mathbb R^n$ into finitely many connected
components.
Each component is a cone. In the following, we discuss unions of these components.

We first introduce the concept of signed distances.
Let $P$ be a hyperplane with a unit normal vector $\nu$ and $p$ be a point on $P$.
Then, the signed distance of $x$ with respect to $\nu$ is defined by
\begin{equation}\label{oriented distance-for plane}
d(x)=
\begin{cases}
|\text{dist}(x,P)|
& \text{if }  \langle x-p ,\nu \rangle >0,\\
0
& \text{if } x \in P,\\
-|\text{dist}(x,P)|
& \text{if } \langle x- p,\nu \rangle < 0.\\
\end{cases}
\end{equation}
Obviously, the signed distance is independent of the choice of $p\in P$.

Fix an integer $ k\ge 2$. 

\begin{definition}\label{def-cone-lipschitz}
Let $P_1, \cdots, P_k$ be $k$ hyperplanes in $\mathbb R^n$ passing the origin 
with mutually distinct normal vectors and 
let $V$ be a Lipschitz infinite cone. 
Then, $V$ is called to be bounded by $P_1, \cdots, P_k$ if 
$$\partial V\subseteq \bigcup_{i=1}^k  P_i.$$  
We call $F_i=\partial V\cap P_i$ a face of $V$. For convenience, we always assume  
$\overline{\partial V \backslash P_i}\neq
\partial V$, for each $i=1,\cdots, k$.
\end{definition}

Next, we put the hyperplanes $P_1, \cdots, P_k$ and the cone $V$ 
as in Definition \ref{def-cone-lipschitz} in a standard position. 
We choose a coordinate system such that, for each $i=1, \cdots, k$,
$$P_i=\{x\in\mathbb R^n:\, x_n=L_i(x')\},$$
for some linear function $L_i$ in $\mathbb R^{n-1}$, and 
$$V=\{x\in\mathbb R^n:\, x_n>g(x')\},$$
for some $g(x')\in \{L_1(x'), L_2(x'),\cdots,L_k(x') \}$.
Then, 
$$\partial V=\{x\in\mathbb R^n:\, x_n=g(x')\}.$$
Next, let $\nu_1, \cdots, \nu_k$ be unit normal vectors of $P_1, \cdots, P_k$, respectively, such that, 
for any $i=1, \cdots, k$,
\begin{equation}\label{eq-InnerNormal}\langle \nu_i, e_n\rangle >0.\end{equation}
The vectors $\nu_1, \cdots, \nu_k$ are called the {\it inner} unit normal vectors associated with the cone $V$.
Moreover, set 
\begin{align*}H_{i}^{1}&=\{x\in\mathbb R^n:\,x_n>L_i(x')\},\\
H_{i}^{-1}&=\{x\in\mathbb R^n:\,x_n\leq L_i(x')\},\end{align*}
and
\begin{equation}\label{eq-definition-cone}V_{(l_1,\cdots,l_k)}=\bigcap_{i=1}^{k}H_{i}^{l_{i}},\end{equation}
where $l_i=1$ or $-1$ for each $i=1,\cdots,k$. Then, the  Lipschitz infinite cone $V$ as in Definition \ref{def-cone-lipschitz}
can be expressed by the union of some $V_{(l_1,\cdots,l_k)}$, i.e., 
\begin{equation}\label{eq-Conecomponet}V= \bigcup V_{(l_1,...,l_k)}, \end{equation}
where the union is over a finite collection of vectors of the form $(l_1,\cdots,l_k)$,  
with $l_i=1$ or $-1$ for each $i=1,\cdots,k$.

It is straightforward to verify the following result. 

\begin{lemma}\label{lemma-cone-relation} Let $V$ be a  Lipschitz infinite cone bounded by $k$ 
hyperplanes $P_1, \cdots, P_k$ with mutually distinct normal vectors in the above setting. 
Then, 

$\operatorname{(i)}$ $V_{(\underbrace{1,\cdots,1}_{k})}\subseteq V$ and 
$V_{(\underbrace{-1,\cdots,-1}_{k})}\bigcap V = \emptyset$;

$\operatorname{(ii)}$  if $(l_1, \cdots, l_k) \leq (m_1, \cdots, m_k)$
and
$V_{(l_1, \cdots, l_k) }\subseteq V,$
then, 
$V_{(m_1, \cdots, m_k) }\subseteq V.$
\end{lemma} 

Here and hereafter, $(l_1, \cdots, l_k) \leq (m_1, \cdots, m_k)$ simply means $l_i \leq m_i$ for each $i=1, \cdots, k$.

\smallskip 

We point out that the interior of $\partial V \bigcap P_i$ may have more than one connected components, 
even for $k\le n$. 

\begin{example}\label{example-4components} Let $P_1, P_2, P_3$ be $3$
linearly independent hyperplanes in $\mathbb R^3$ satisfying \eqref{eq-InnerNormal} for $n=3$ 
and $V\subseteq \mathbb R^{3}$ be an infinite cone given by  
$$V=V_{(1,1,1)}\bigcup V_{(-1,1,1)} \bigcup V_{(1,-1,1)} \bigcup V_{(1,1,-1)}.$$
The interior of each $\partial V \bigcap P_i$  has two connected components.\end{example}

\begin{definition}\label{def-cone-same-relation}
Let $V_1, V_2\subseteq \mathbb R^{n}$ be two infinite cones bounded 
by two sets of 
hyperplanes
$P_{1,1}, \cdots, P_{1,k}$ and $P_{2,1}, \cdots, P_{2,k}$, respectively. Then,  $V_1$ and $V_2$ 
are said to {\it satisfy the same relation} if the collections of  
$(l_{1,1}, \cdots, l_{1,k})$ in \eqref{eq-Conecomponet} for $V_1$ 
and of $(l_{2,1}, \cdots, l_{2,k})$ in \eqref{eq-Conecomponet} for $V_2$ are identical.
\end{definition}

Let $V \subseteq \mathbb R^{n} $ be an infinite cone as in Definition \ref{def-cone-lipschitz}
and $\nu_1, \cdots, \nu_k$ be its inner unit normal vectors.
Now, we require $k\le n$ and $\nu_1, \cdots, \nu_k$ are linearly independent. 
By a rotation,  we assume $V=
V_{k} \times \mathbb R^{n-k}$, with $V_k\subset \mathbb R^k$. 
Then for any $x $, the projection from $x$ to $\mathbb R^k\times\{0\}$ can
be uniquely determined by
$d_{1}(x),...,d_k(x)$, where $d_i$ is the signed distance form $x$ to $P_i$ with respect to $\nu_i$.
With such a 
one-to-one correspondence between $x\in \mathbb R^k\times \{0\}$ and $(d_1, \cdots, d_k)$,
we rewrite the solution of
\eqref{eq-MainEq}-\eqref{eq-MainBoundary} for $\Omega=V$ in Theorem
\ref{existence in cone} as
\begin{equation}\label{eq-Solution-Cone-d-coordinates-nd2}f_{V}(d_1(x) ,...,
d_k(x))=u_{V}(x).\end{equation}
If we treat
$\nu_1, \cdots, \nu_k$  as column vectors,
then the matrix $(\nu_1, \cdots, \nu_k)$ is a $k\times n$ matrix. By the linear independence,
the $k\times k$ matrix $(\nu_1, \cdots, \nu_k)^T(\nu_1, \cdots, \nu_k)$ is invertible. Set
\begin{equation}\label{eq-LinearIndependence}
\sigma(P_1, \cdots, P_k)=\|((\nu_1, \cdots, \nu_k)^T(\nu_1, \cdots, \nu_k))^{-1}\|.\end{equation}
We also note
$$\{d_1>0, \cdots, d_k>0\}\subset V,\quad \{d_1<0, \cdots, d_k<0\}\cap V=\emptyset.$$

For $k=n$, 
$\bigcap_{j \neq i} P_j$ is an {\it edge} of $V$, which is transversal to $P_i$. In the following,
we always denote by $\mu_i$ the unique unit vector such that
\begin{equation}\label{sec-Definition-mu}
\mu_i\in\bigcap_{j \neq i} P_j, \quad \langle \mu_i, \nu_i\rangle>0.
\end{equation}
Then, we can check 
$$\{d_1>0, \cdots, d_n>0\}=\{t_1\mu_1+\cdots+t_n\mu_n:\, t_1>0, \cdots, t_n>0\},$$
and
$$\{d_1<0, \cdots, d_n<0\}=\{t_1\mu_1+\cdots+t_n\mu_n:\, t_1<0, \cdots, t_n<0\}.$$
In Example \ref{example-4components}, edges consist of three lines, instead of three rays 
we usually anticipate. 

\smallskip

Next, we turn our attention to domains.
We always assume that $\Omega$ is a bounded Lipschitz domain, with $x_0\in\partial\Omega$.
We can define the {\it tangent cone} of $\Omega$ at $x_0$ by blowing up
$\Omega$ near $x_0$. We will not present such a definition for general domains. In the following,
we describe an equivalent way to construct tangent cones
for domains bounded by finitely many $C^{1,1}$-hypersurfaces.

\smallskip

Let $S$ be a $C^{1,1}$-hypersurface in a neighborhood of $x_0\in S$,
and $\nu$  be a continuous unit normal vector field over $S$ near $x_0$.
Here and hereafter, we always assume $x_0$ is an interior point of $S$.
For any $x$ close to $x_0$, set $p_{x}$ to be the point on $S$ with the least distance to $x$.
Then, the {\it signed distance} of $x$ to $S$ with respect to $\nu$ is defined by
\begin{equation}\label{oriented distance}
d(x)=
\begin{cases}
|xp_x|
& \text{if }  \langle x-p_x ,\nu_{p_x} \rangle >0,\\
0
& \text{if } x \in S,\\
-|xp_x|
& \text{if } \langle x- p_x,\nu_{p_x} \rangle < 0.\\
\end{cases}
\end{equation}

Fix an integer $k\ge 2$. 

\begin{definition}\label{def-Domain-lipschitz}
Let $S_1, \cdots, S_k$  be $k$ $C^{1,1}$-hypersurfaces passing $x_0$, with 
mutually distinct normal vectors of tangent planes of
$S_1, \cdots, S_k$  at $x_0$, and let $\Omega$ be a bounded Lipschitz domain in $\mathbb R^n$ with 
$x_0\in\partial\Omega$. Then, $\Omega$ is called to be bounded by $S_1, \cdots, S_k$ near $x_0$
if, for some $R>0$,  
$$\partial \Omega \bigcap B_{R}(x_0)\subseteq \bigcup_{i=1}^k  S_i.$$  
We always assume $\overline{ \partial \Omega \bigcap B_{r}(x_0)\backslash S_i}\neq
\partial \Omega \bigcap B_{r}(x_0)$, for any $r \leq R$ and any $i=1,\cdots,k$. \end{definition}

Let $S_1, \cdots, S_k$ and $\Omega$ be as in Definition \ref{def-Domain-lipschitz}. 
We choose a coordinate system such that, 
for  each $i=1, \cdots, k$,
$$S_i\cap B_{R}(x_0)=\{x\in B_{R}(x_0):\, x_n=f_i(x')\},$$
for some $C^{1,1}$-function $f_i$ in $B_{R}'(x_0')$, and 
$$\Omega\cap B_{R}(x_0)=\{x\in B_{R}(x_0):\, x_n>g(x')\},$$
for some $g(x')\in \{f_1(x'), f_2(x'),\cdots,f_k(x') \}$.
Then, 
$$\partial \Omega\cap B_{R}(x_0)=\{x\in B_{R}(x_0):\, x_n=g(x')\}.$$
Next, let $\nu_1, \cdots, \nu_k$ be unit normal vectors of 
the tangent planes of $S_1, \cdots, S_k$ at $x_0$ 
such that, for any $i=1, \cdots, k$,
$$\langle \nu_i, e_n\rangle >0.$$ 
The vectors $\nu_1, \cdots, \nu_k$ are called the {\it inner} unit
normal vectors associated with $\partial\Omega$ at $x_0$.
Moreover, set 
\begin{align*}S_{i}^{1}&=\{x\in B_{R}(x_0):\, x_n>f_i(x')\},\\
S_{i}^{-1}&=\{x\in B_{R}(x_0):\, x_n\leq f_i(x')\},\end{align*}
and 
$$\Omega_{(l_1,\cdots,l_k)}=\bigcap_{i=1}^{k}S_{i}^{l_{i}},$$
where $l_i=1$ or $-1$ for each $i=1,\cdots,k$. Then, $\Omega\cap B_{R}(x_0)$
can be expressed by the union of some $\Omega_{(l_1,...,l_k)}$, i.e., 
\begin{equation}\label{eq-Domaincomponet}\Omega\bigcap B_{R}(x_0)= \bigcup \Omega_{(l_1,\cdots,l_k)}, \end{equation}
where the union is over a finite collection of vectors of the form $(l_1,\cdots,l_k)$,  with $l_i=1$ or $-1$ for each $i=1,\cdots,k$.

It is straightforward to verify the following result. 

\begin{lemma}\label{lemma-domain-relation} Let $\Omega$ be a bounded Lipschitz domain bounded by $k$ 
$C^{1,1}$-hypersurfaces $S_1, \cdots, S_k$ in $B_{R}(x_0)$ 
for some $x_0\in\partial\Omega$ and 
some $R>0$ in the above setting, with mutually distinct normal vectors of the tangent planes of 
$S_1, \cdots, S_k$ at $x_0$. 
Then, 

$\operatorname{(i)}$ $\Omega_{(\underbrace{1,\cdots,1}_{k})}\subseteq \Omega$ and 
$\Omega_{(\underbrace{-1,\cdots,-1}_{k})}\bigcap \Omega = \emptyset$;

$\operatorname{(ii)}$  if $(l_1,\cdots,l_k) \leq (m_1,\cdots,m_k)$
and
$\Omega_{(l_1,\cdots,l_k) }\subseteq \Omega,$
then, 
$\Omega_{(m_1,\cdots,m_k) }\subseteq \Omega.$
\end{lemma} 

For domains as in Definition \ref{def-Domain-lipschitz}, we can characterize their tangent cones easily. 
Let $\Omega\subset \mathbb{R}^{n}$ be a bounded Lipschitz domain bounded by $k$ $C^{1,1}$-hyperplanes 
$S_1, \cdots, S_k$ in a neighborhood of $x_0\in\partial\Omega$, satisfying \eqref{eq-Domaincomponet}.
Suppose that the tangent plane $P_i$ of $S_i$ at $x_0$ is
given by $x_n=L_i(x')$, for each $i=1, \cdots, k$. Then,  the tangent cone
$V_{x_0}$ of $\Omega$ at $x_0$ is the infinite cone given by
\begin{equation}\label{tangent cone}
V_{x_0}= \bigcup V_{(l_1,\cdots,l_k)},
\end{equation}
where $V_{(l_1,\cdots,l_k)}$ is defined in \eqref{eq-definition-cone}, 
and the union over $(l_1,\cdots,l_k)$ in \eqref{tangent cone} is the same as that in \eqref{eq-Domaincomponet}.

To end this section, we make a remark concerning the necessity  to put cones and domains in standard positions. 
Recall that we start with a bounded Lipschitz domain which is bounded by finitely many $C^{1,1}$-hypersurfaces near 
a boundary point. By putting this domain in the standard position, we can characterize its tangent cone 
easily, as demonstrated in the proceeding paragraph. The advantage here is to simply take the union over the same set 
of finitely elements in the form $(l_1,\cdots,l_k)$,  with $l_i=1$ or $-1$ for each $i=1,\cdots,k$.
In Section \ref{sec-Singular-1}, we will construct more sets  bounded by spheres along the same line.

\section{Solutions in Different Cones}\label{sec-DifferentCones}

In section, we compare solutions in different cones.

We first prove a basic result. We point out that, for an infinite cone with a Lipschitz boundary,
there always exists a uniform infinite exterior cone at each point of its boundary.

\begin{lemma}\label{slu-dunder-affine-tr}
Let $V_1, V_2\subseteq \mathbb R^{n}$ be two infinite cones with Lipschitz boundaries 
and vertices at the origin and
$u_i$ be the  nonnegative solution of \eqref{eq-MainEq}-\eqref{eq-MainBoundary}
for $\Omega=V_i$, $i=1,2$. Assume there exists a linear transform $T =O_1 AO_2$, for some
$O_1, O_2 \in O(n)$ and $A \in GL(n)$, such that $TV_1= V_2 $. Then, there exist a
positive constant $\epsilon_0$,  depending only on $n$
and the size of  exterior
cones, such that,  if $\|A-I\| \leq
\epsilon_0$, then for any $x \in V_2 $,
$$(1+C\|A-I\|)^{-1} u_2 (x)\leq \ u_1 (T^{-1}x) \leq (1+C\|A-I\|) u_2 (x),$$
where  $C$ is a positive constant depending only on $n$ and the size of  exterior
cones.
\end{lemma}

As noted before, $\|\cdot\|$ is the norm of matrices, considered as transforms in the Euclidean spaces.

\begin{proof}
First, $u_1$ satisfies
$$\Delta u_1  = \frac14n(n-2) u_1^{\frac{n+2}{n-2}}\quad\text{in }V_1.$$
Set $ \widetilde{u}_1 (x) = u_1 (T^{-1}x)$, for any $x\in V_2$, and $(a_{ij}) =
 AA^{T}$. Then,
$$a_{ij} \widetilde{u}_{1\, ij} =
\frac14n(n-2) \widetilde{u}_1^{\frac{n+2}{n-2}}\quad\text{in }V_2,$$
and hence
$$-\Delta \widetilde{u}_1  +  \frac14n(n-2) \widetilde{u}_1^{\frac{n+2}{n-2}}
 =(a_{ij}-\delta_{ij})\widetilde{u}_{1\,ij}\quad\text{in }V_2.$$
Therefore,
\begin{align*}
&-\Delta [(1+C\|A-I\|)\widetilde{u}_1 ] +  \frac14n(n-2) [
(1+C\|A-I\|)\widetilde{u}_1]^{\frac{n+2}{n-2}}\\
&\qquad=\frac14n(n-2) [ (1+C\|A-I\|)^{\frac{n+2}{n-2}} -(1+C\|A-I\|) ]
\widetilde{u}_1^{\frac{n+2}{n-2}} \\
&\qquad\qquad+ (1+C\|A-I\|)  (a_{ij} -
\delta_{ij})\widetilde{u}_{1\,ij}.
\end{align*}
We first choose $\epsilon_0 <{1}/{10}$. Then,
$$\|A^{-1}\| =\|(I-(I-A))^{-1}\| \leq
\frac{1}{1-\|A-I\|},$$ and
$$ (1-\|A-I\| )dist( x,\partial V_2) \leq dist(T^{-1} x,\partial V_1) \leq (1+\|A-I\| )
dist( x,\partial V_2).$$
Next, we note
$$|\delta_{ij}-a_{ij}| \leq \|A^{T}A-I\| \leq 3\|A-I\|,$$
and, by Lemma \ref{lemma-ExteriorCone} and Remark \ref{rmrk-delta},
$$C'\leq dist(x, \partial V_i)^{\frac{n-2}{2}}u_i (x) \leq 2^{\frac{n-2}{2}}.$$
Lemma \ref{lemma-GradientEstimates} implies
$$|(\delta_{ij}-a_{ij})\widetilde{u}_{1\,ij}| \leq C_1\|A-I\| dist(x,\partial
V_2)^{-\frac{n+2}{2}}.$$
We also have
\begin{align*}
&\,[ (1+C\|A-I\|)^{\frac{n+2}{n-2}} -(1+C\|A-I\|) ] \widetilde{u}_1^{\frac{n+2}{n-2}} \\
\geq &\, \frac{4}{n-2}C\|A-I\|\widetilde{u}_1^{\frac{n+2}{n-2}}  \geq C\|A-I\| C_{2} (n)
dist(x,\partial
V_2)^{-\frac{n+2}{2}}.
\end{align*}
By taking $C = C_3>2 {C_{1}}/{C_2}$ and requiring $\epsilon_0 \leq
{1}/{C_3}$, we obtain
$$-\Delta [(1+C\|A-I\|)\widetilde{u}_1 ] +  \frac14n(n-2) [
(1+C\|A-I\|)\widetilde{u}_1]^{\frac{n+2}{n-2}} \geq 0.$$
Therefore, by the maximum principle, we get
$$u_2 (x)\leq \ u_1 (T^{-1}x)(1+C\|A-I\|).$$
Interchanging the position of $u_1$ and $u_2$, we also have
$$ u_1 (T^{-1}x) \leq (1+C\|A^{-1}-I\|) u_2 (x)\leq (1+C\|A-I\|) u_2 (x).$$
In summary, by choosing $\epsilon_0 \leq \min\{{1}/{10},{1}/{C_3}\}$, we get the desired estimate.
\end{proof}

Next, we discuss cones introduced in Definition \ref{def-cone-lipschitz}
and compare different $f_{V}$ introduced in \eqref{eq-Solution-Cone-d-coordinates-nd2}
in different cones.
Lemma \ref{lemma-AnisotropicEstimate} plays an essential role.

\begin{lemma}\label{slu-di-dn dunder-affine-tr-genral}
For a fixed $2\le k\le n$,
let $V_1, V_2\subseteq \mathbb R^{n}$ be two infinite cones satisfying the same relation
as in Definition \ref{def-cone-same-relation} in terms of  linearly independent hyperplanes
$P_{1,1}, \cdots, P_{1,k}$ and $P_{2,1}, \cdots, P_{2,k}$, respectively,
and $u_i$
be the unique nonnegative solution of \eqref{eq-MainEq}-\eqref{eq-MainBoundary} for
$\Omega=V_i$, with $f_{V_i}$ given by \eqref{eq-Solution-Cone-d-coordinates-nd2}, for $i=1, 2$.
Assume there exists a linear transform $T =O_1 AO_2$,
for some $O_1, O_2 \in {O}(n)$ and  $A \in {GL}(n)$, such that $TV_ 1= V_2$ and
$T(\partial V_1\cap P_ {1,j})= \partial V_2\cap P_{2,j}$,  for $j=1,\cdots, k$.
For the positive constant $ \epsilon_{0}$  determined in Lemma \ref{slu-dunder-affine-tr},
if $\|A-I\| \leq \epsilon_{0}$, then
$$(1+C\|A-I\|)^{-1} f_{V_2}(d_1,\cdots, d_k)\leq f_{V_1}(d_1, \cdots, d_k)  \leq (1+C\|A-I\|) 
f_{V_2}(d_1,\cdots, d_k),$$
where  $C$ is some positive constant depending
only on $n$ and $\sigma(P_{i,1},...,P_{i,k})$ as defined in
\eqref{eq-LinearIndependence}.
\end{lemma}

\begin{proof} We consider only the case $k=n$. 
Fix $d_{1},d_{2},\cdots,d_n$ such that the corresponding $x\in V_2$.
By Lemma \ref{slu-dunder-affine-tr}, we have
\begin{equation}\label{eq-interior-cone-estimate1}
(1+C\|A-I\|)^{-1} u_2 (x)\leq u_1 (T^{-1}x) \leq (1+C\|A-I\|) u_2 (x).\end{equation}
Note, for each $j=1, \cdots, n$,
$$|\widehat{d}_{j}- d_j| \leq \|A-I\| | d_j|,$$
where $\widehat{d}_{j}$ is the signed distance from $T^{-1}x$ to $P_{1,j}$.
Let $x^*$ be the point in $V_1$ such that its signed distance to $P_{1,j}$ is $d_j$. 
Then, 
$$|(T^{-1}x)_i- x^*_i| \leq C\|A-I\| | x^*_i|.$$
By Lemma \ref{lemma-AnisotropicEstimate}, in particular \eqref{eq-AnisotropicDifference}, 
we have 
\begin{equation}\label{eq-interior-cone-estimate2}
|{u}_1(T^{-1}x) -{u}_1(x^*) | \leq C \|A-I\|{u}_1(x^*).
\end{equation}
By combining with \eqref{eq-interior-cone-estimate1} and \eqref{eq-interior-cone-estimate2}, we obtain 
$$(1+C\|A-I\|)^{-1} u_2(x)\leq u_1(x^*)  \leq (1+C\|A-I\|) u_2(x).$$
This implies the desired result. \end{proof}

Next, we express the condition in Lemma \ref{slu-di-dn dunder-affine-tr-genral}
in terms of the inner unit normal vectors.

Let $V$ be an infinite cone  as introduced in Definition \ref{def-cone-lipschitz}
and $\nu_1, \cdots, \nu_k$ be inner unit normal vectors of $V$.
Consider the $n\times k$ matrix
\begin{equation}\label{eq-Definition-N}N=(\nu_1, \cdots,\nu_k).\end{equation}
If $k=n$, let $\mu_1, \cdots, \mu_n$ be the vectors introduced in \eqref{sec-Definition-mu}.
Then, for any $j=1, \cdots, n$,
$$\langle \nu_l, \mu_{j}\rangle=0\quad\text{for }l\neq j,$$
and
$$\langle \nu_j, \mu_{j}\rangle>0.$$
As a consequence, for any $j=1, \cdots, n$,
$$N^{T}\mu_{j} =\langle \nu_j, \mu_{j}\rangle e_j,$$
where $e_j$ is the unit vector along the $x_j$-axis. This implies, in particular, that
$N^T\mu_j$ and $e_j$, or $\mu_j$ and $N^{-T}e_j$,  are along the same direction.

We have the following result.

\begin{theorem}\label{slu-k n1-n2 under tr-general}
For a fixed $2\le k\le n$,
let $V_1, V_2\subseteq \mathbb R^{n}$ be two infinite cones satisfying the same relation
as in Definition \ref{def-cone-same-relation} in terms of  linearly independent hyperplanes
$P_{1,1}, \cdots, P_{1,k}$ and $P_{2,1}, \cdots, P_{2,k}$, respectively,
and $u_i$
be the unique nonnegative solution of \eqref{eq-MainEq}-\eqref{eq-MainBoundary} for
$\Omega=V_i$, with $f_{V_i}$ given by \eqref{eq-Solution-Cone-d-coordinates-nd2}, for $i=1, 2$.
For the positive constant $ \epsilon_{0}$
determined in Lemma \ref{slu-dunder-affine-tr},
if
$$3n \|(N_{1}^{T}N_1)^{-1}\|\|N_1-N_2\| \leq \epsilon_0,$$
then,
$$(1+C\|N_1-N_2\|)^{-1} f_{V_2}(d_1,..., d_k)\leq f_{V_1}(d_1, ..., d_k)  \leq
(1+C\|N_1-N_2\|)f_{V_2}(d_1, ..., d_k),$$
where $N_1$ and $N_2$ are the matrices defined as in \eqref{eq-Definition-N}
for $V_1$ and $V_2$, and  $C$ is a positive constant depending
only on $n$ and $\sigma(P_{i,1},...,P_{i,k})$ as defined in
\eqref{eq-LinearIndependence}.
\end{theorem}

\begin{proof}
We first consider $k=n$. Set
\begin{equation}\label{eq-Definition-A}A= (
\mu_{2,1},... ,   \mu_{2,n} )( \mu_{1,1},... ,   \mu_{1,n})^{-1}.\end{equation}
Then, the linear transform given by $T =A$ satisfies $T\mu_{1,j}=\mu_{2,j}$, for
$j=1, \cdots, n$. Since $V_1, V_2$ satisfy the same relation, then $TV_ 1= V_2 $
and $TF_ {1,j}= F_{2,j}$, for $j=1,\cdots, n$.
We now verify  $\|A-I\| \leq
\epsilon_{0}$.
First, we have \begin{equation}\label{e_l_i--e_i} \mu_{i,j}=(N_{i}^{T})^{-1}\langle
\nu_{i,j}, \mu_{i,j}\rangle e_j.\end{equation}
Then,
\begin{align*}&\,|(N_{1}^{T})^{-1}\langle \nu_{1,j}, \mu_{1,j}\rangle e_j
-(N_{2}^{T})^{-1}\langle \nu_{1,j}, \mu_{1,j}\rangle e_j| \\
\leq &\, \|(N_{2}^{T})^{-1}\|
\|N_2^{T} -N_1^{T} \| |(N_{1}^{T})^{-1}\langle \nu_{1,j}, \mu_{1,j}\rangle e_j| \\
\leq  &\, \|N_2^{-1}\|  \|N_2^{} -N_1 \| \\
\leq &\, \frac{\|N_1^{-1}\|}{1-\|N_1^{-1}\| \|N_2^{} -N_1 \|} \|N_2 -N_1 \|.
\end{align*}
Therefore, if $\|N_1^{-1}\| \|N_2^{} -N_1 \| \leq {1}/{10} $, we have
\begin{align*}
| \mu_{1,j} -  \mu_{2,j}| &\leq |\mu_{1,j}-(N_{2}^{T})^{-1}\langle \nu_{1,j},
\mu_{1,j}\rangle e_j|  + |(N_{2}^{T})^{-1}\langle \nu_{1,j}, \mu_{1,j}\rangle e_j
-\mu_{2,j}|\\
&\leq2 |\mu_{1,j}-(N_{2}^{T})^{-1}\langle \nu_{1,j},
\mu_{1,j}\rangle e_j|\\
& \leq 3 \|N_1^{-1}\| \|N_2^{} -N_1 \|,
\end{align*}
where, for the estimate of $(N_{2}^{T})^{-1}\langle \nu_{1,j}, \mu_{1,j}\rangle e_j
-\mu_{2,j}$,
we use the facts that $(N_2^T)^{-1}e_j$ and $\mu_{2,j}$
are on the same ray and $|\mu_{2,j}|=|\mu_{1,j}|$. Hence,
\begin{align*}
\|A-I\| &\leq  \|\left( \mu_{2,1}-\mu_{1,1},\cdots ,  \mu_{2,n} -\mu_{1,n}\right)\|   \|(
\mu_{1,1}, \cdots,   \mu_{1,n})^{-1} \|\\
& \leq     3\sqrt{n} \|N_1^{-1}\| \|N_2^{} -N_1 \| \|( \mu_{1,1},... ,   \mu_{1,n})^{-1}\|.
\end{align*}
By \eqref{eq-Estimate-E}, we obtain
\begin{align*}
\|( \mu_{1,1},... ,   \mu_{1,n})^{-1} \|
\leq \sqrt{n}\| N_1^{-1}\|,
\end{align*}
and then
$$\|A-I\|  \leq 3n  \|N_1^{-1}\|^{2} \|N_2^{} -N_1 \|.$$
Hence,
$\|A-I\| \leq \epsilon_{0}$. Then, we can apply Lemma \ref{slu-di-dn dunder-affine-tr-genral}.

Next, we consider $2\le k<n$. Set $E=$ span$\{ \nu_{1,1}, \nu_{1,2},\cdots,\nu_{1,k}\}$
and let $e_{k+1}, \cdots, e_n$ be the orthonormal basis of the orthogonal complement of $E$.
Consider the matrices
$$\overline{N}_{1}=(\nu_{1,1},\cdots, \nu_{1,k},e_{k+1},...,e_{n}), \quad
\overline{N}_{2} =(\nu_{2,1},\cdots, \nu_{2,k},e_{n+1},...,e_{n}).$$
Then, we have
$$\|\overline{N}_{1}^{-1}\|^2=\|(\overline{N}_{1}^{T}\overline{N}_{1})^{-1}\|=\|(N_1^{T}N_1)^{-1}\|,$$
and
$$\|N_1-N_2  \| =\| \overline{N}_1-\overline{N}_2   \|.$$
Let $P_j$ be the hyperplane passing the origin with its unit normal vector $e_j$, $j= k+1,..., n$.
For a unification, we write $P_{i,j}=P_j$, for $i=1, 2$ and $j=k+1, \cdots, n$.
Let $\overline{V}_i$ be the {\it convex} infinite cone enclosed by the $n$ hyperplanes
$P_{i,1},...,P_{i,k}$, $P_{i,k+1},...,P_{i,n}$ and 
$\widehat{V}_i$ be the {\it convex} infinite cone enclosed by the $k$ hyperplanes
$P_{i,1},...,P_{i,k}$. 
Denote by $F_{i,1}, \cdots, F_{i,n}$ the faces of $\overline{V}_{i}$ and by $\mu_{i,j}$ the unit vector along
$\bigcap_{l\neq j} F_{i,l}$. Then, for $i=1,2$,
$$\overline{V}_i =\{t_{i,1}  \mu_{i,1} +\cdots+t_{i,n} \mu_{i,n}:\, t_{i,j}>0
\text{ for }j=1,...,n\}.$$ 
Define $A$ by \eqref{eq-Definition-A}. Hence,
the affine transform given by $T =A$ satisfies $T\mu_{1,j}=\mu_{2,j}$, for
$j=1, \cdots, n$, and therefore
$TV_ 1= V_2 $. 
By Lemma \ref{slu-di-dn dunder-affine-tr-genral}, we have the desired estimate.
\end{proof}

\section{Solutions near Singular Points}
\label{sec-Singular-1}

In this section, we study the
asymptotic
expansions near singular points intersected by $C^{1,1}$-hypersurfaces in
$\mathbb{R}^{n}$ and derive an optimal estimate. The discussion relies essentially
on the conformal invariance of the equation \eqref{eq-MainEq}.

We first prove some results concerning the intersection of spheres.

\begin{lemma}\label{two ball inters}
Let $\partial B_{R_i}(o_i)$ be two circles in $\mathbb R^2$, $i=1,2,$
intersecting  at two points $p$ and  $q$. Suppose the angle between $po_1$ and $po_2$
is $\alpha$ for some $\alpha \in (0,\pi)$. Then,
$$|pq|=\frac{2R_1R_2\sin\alpha}{\sqrt{R_{1}^{2}+R_{2}^{2}-2R_1R_2\cos\alpha}}.$$
 \end{lemma}

\begin{proof}
This follows by a  straightforward calculation. We simply note $o_1o_2\perp pq$ and 
$|o_1o_2|=\sqrt{R_{1}^{2}+R_{2}^{2}-2R_1R_2\cos\alpha}$. 
\end{proof}

\begin{lemma}\label{n ball inters}
Let $\partial B_{1}(o_i)$ be $n$ unit spheres in $\mathbb{R}^{n}$, $i= 1,...,n,$
intersecting at a point $p$,
and the inner unit normal vector $\nu_i$ of $\partial B_{1}(o_i)$ at $p$
be linearly independent. Then, $\partial B_{1}(o_i)$ intersect
at another point $q$, and
$$|pq| > \frac{|\det N|}{2^{n-2}},$$
where $N$ is the matrix $(\nu_1, \nu_2,\cdots,\nu_n)$.
\end{lemma}

\begin{proof}
We  prove by induction.

For $n=2$, by Lemma \ref{two ball inters}, $\partial B_{1}(o_i)$ intersect
at two points $p$ and $q$, and
$$|pq|=\frac{2\sin\alpha}{\sqrt{2-2\cos\alpha}}>  \sin \alpha =|\det
N|,$$
where $\alpha$ is the angle between $\nu_1$ and $\nu_2$.
Hence, the desired result holds for $n=2$. In fact, these two circles intersect at exactly two points. 

Now suppose the result holds for $n=k$, for some $k \geq 2$,
and we consider $k+1$. Without loss of generality, we
assume $p=0$ and
$$\aligned\nu_i &=(\nu^{i}_1,\cdots,\nu^{i}_k,0)^{T}\quad\text{for }i=1,...,k,\\
\nu_{k+1}&=(\nu^{k+1}_{1},\cdots,\nu^{k+1}_{k},\nu^{k+1}_{k+1} )^{T}\quad\text{with }\nu^{k+1}_{k+1}>0.
\endaligned$$
Under this setting, the $(k+1)$-th coordinate of $o_i$ is zero,
for $i=1, \cdots, k$.
Write  $\nu_i'=(\nu^{i}_1,\cdots,\nu^{i}_k)^{T}$ and $N_k=(\nu_1', \cdots,\nu_k')$, a
$k\times k$ matrix.  Then,
$$|\det N|=|\det N_{k}|\nu^{k+1}_{k+1}.$$
Write  $B_i'=B_{1}(o_i)\cap\{x_{k+1}=0\}$ and regard it as a ball in $\mathbb R^k=\mathbb R^k\times \{0\}$.
By induction, we have
$$\bigcap_{i=1}^k\partial B_i'=\{p,q'\},$$ and
$$|pq'| > \frac{|\det N_k|}{2^{k-2}}.$$
Without loss of generality, we assume $q'=(2a, 0,...,0)$, with $a>|\det
N_k|/{2^{k-1}}$. It is straightforward to vefify 
$$\bigcap_{i=1}^k\partial B_{1}(o_i)
=\{(x_1,0_{\mathbb R^{k-1}},x_{k+1}):\, (x_1-a)^{2}+x_{k+1}^{2}=a^{2}\}:=C_1.$$
Now we have
\begin{align*}
&\, \partial B_{1}(o_{k+1}) \bigcap   \{(x_1,0_{\mathbb R^{k-1}},x_{k+1}):\,x_1 \in \mathbb{R}, x_{k+1}
\in \mathbb{R}\}\\ =   &\,
\{(x_1,0_{\mathbb R^{k-1}},x_{k+1}):\,(x_1-\nu^{k+1}_{1})^{2}+(x_{k+1}-\nu^{k+1}_{k+1})^{2}=
(\nu_{1}^{k+1 })^2+ (\nu_{k+1}^{k+1 })^2\} := C_2.
\end{align*}
Note $0 \in C_1\bigcap C_2 $.
By Lemma \ref{two ball inters}, there exists a point $q \in C_1\bigcap C_2 $ such that
$$|pq|=\frac{2aR\sin\alpha}{\sqrt{a^{2}+R^{2}-2aR\cos\alpha}},$$
where $R=\sqrt{(\nu_{1}^{k+1 })^2+ (\nu_{k+1}^{k+1 })^2}$, $\cos\alpha={\nu_{1}^{k+1 }}/{R}$
and $\sin\alpha={\nu_{k+1}^{k+1 }}/{R} $.
Therefore, with $a\le 1$ and $R\le 1$, 
$$|pq|>a\nu^{k+1}_{k+1} > \frac{|\det N_{k}|\nu^{k+1}_{k+1} }{2^{k-1}} =
\frac{|\det N|}{2^{k-1}}.$$
This implies the desired result for $n=k+1$.
\end{proof}

In fact, the proof above demonstrates that these spheres intersect at exactly two points.

In general for some $k\le n$,
let $\partial B_{1}(o_i)$ be $k$ unit spheres in $\mathbb{R}^{n}$, $i= 1,\cdots,k,$
intersecting at a point $p$,
such that the inner unit normal vectors  $\nu_i$ of $\partial B_{1}(o_i)$ at $p$
are linearly independent.
Then, $\partial B_{1}(o_i)$ intersects
at an $(n-k)$-dimensional sphere, with its radius
$$r > \frac{\sqrt{\det ( N^{T}N)}}{2^{k-2}},$$
where $N$ is the matrix $(\nu_1, \cdots, \nu_k)$. Here, the 0-dimensional sphere consists of
two points.
In particular,
$\partial B_{1}(o_i)$ intersect
each other at at least two points
$p$ and $q$, and
$$|pq| > \frac{\sqrt{\det ( N^{T}N)}}{2^{k-2}}.$$
Without loss of generality, we assume $p=(a,0,\cdots,0)$, and $q=(-a,0,\cdots,0)$,
where
$$a=\frac{|pq|}{2} > \frac{\sqrt{\det ( N^{T}N)}}{2^{k-1}}.$$

Next, we set, for each $i=1, \cdots, k$, 
\begin{align*}\widetilde{B}_{i}^{1}&=\{x\in\mathbb R^n:\, x\in B_{1}^{n}(o_i)\},\\
\widetilde{B}_{i}^{-1}&=\{x\in\mathbb R^n:\, x\notin B_{1}^{n}(o_i)\},\end{align*}
and 
\begin{equation}\label{eq-definition-tildeB}
\widetilde{B}_{(l_1,...,l_k)}=\bigcap_{i=1}^{k}\widetilde{B}_{i}^{l_{i}},\end{equation}
where $l_i=1$ or $-1$, for each $i=1,\cdots,k$. 
Then, we take
\begin{equation}\label{eq-interiorballcomponet}\widetilde{B}= \bigcup \widetilde{B}_{(l_1,\cdots,l_k)} \end{equation}
where the union is over a fixed finite set of vectors 
$(l_1,\cdots,l_k)$,  with $l_i=1$ or $-1$ for each $i=1,\cdots,k$. 
We require that 
$\widetilde{B}$ is a Lipschitz domain. 








Similarly, we set, for each $i=1, \cdots, k$, 
\begin{align*}\widehat{B}_{i}^{1}&=\{x\in\mathbb R^n:\, x\notin \overline{B_{1}^{n}(o_i)}\},\\
\widehat{B}_{i}^{-1}&=\{x\in\mathbb R^n:\,x\in \overline{B_{1}^{n}(o_i)}\},\end{align*}
and 
\begin{equation}\label{eq-definition-hatB}\widehat{B}_{(l_1,\cdots,l_k)}=\bigcap_{i=1}^{k}\widehat{B}_{i}^{l_{i}},
\end{equation}
where $l_i=1$ or $-1$, for each $i=1,\cdots,k$. Then, we take
\begin{equation}\label{eq-exteriorballcomponet}\widehat{B}= \bigcup \widehat{B}_{(l_1,\cdots,l_k)}, \end{equation}
where the union is over the same set of vectors  $(l_1,\cdots,l_k)$ as in \eqref{eq-interiorballcomponet}.

\smallskip

In the following, we transform  $\widetilde{B}$ and $\widehat{B}$ conformally to infinite cones. 
To this end, we embed $\mathbb{R}^{n}$ into $\mathbb{R}^{n+1}$ as $\mathbb{R}^{n}\times \{0\}$. Set
$$S^n(a)=\{(x_1,\cdots, x_n,x_{n+1}):\,x_{1}^{2}+\cdots+x_{n+1}^{2}=a^{2}\}.$$
Consider the following three conformal transforms in $\mathbb R^{n+1}$. We always write
$x=(x_1, \cdots, x_n)\in\mathbb R^n$.
Set
$$T_1(x,0)=\left(\frac{2a^2 x_1}{a^2+|x|^2},\cdots,\frac{2a^2 x_n}{a^2+|x|^2},\frac{a(a^2-|x|^2)}{a^2+|x|^2}\right).$$
Then, $T_1$ is the inverse transform of the stereographic projection
which lifts $\mathbb{R}^{n}\times \{0\}$ to $S^{n}(a)$. Next, set
$$T_2(x_1,...,x_n,x_{n+1})=(-x_{n+1},x_2,\cdots,x_n,x_1).$$
Then, $T_2$ is an orthogonal transform which
rotates the $x_1x_{n+1}$-plane by ${\pi}/{2}$ is counterclockwisely.
Last, set
$$T_3(x,x_{n+1})=\left(\frac{ax_1}{a+x_{n+1}},\cdots,\frac{ax_n}{a+x_{n+1}},0\right).$$
Then, $T_3$ is the stereographic projection which
transforms $S^{n}(a)\backslash(0_{\mathbb R^n},-a)$ onto $\mathbb{R}^{n}\times \{0\}$.
Next, we view $T_a=(T_3T_2T_1)|_{\mathbb R^n\times \{0\}}$ as a map in $\mathbb R^n$,
which is given by
\begin{equation}\label{eq-ConformalTransform}T_a x=\left(\frac{-a(a^2-|x|^2)}{a^2+2ax_1+|x|^2},
\frac{2a^2 x_2}{a^2+2ax_1+|x|^2},\cdots,
\frac{2a^2 x_n}{a^2+2ax_1+|x|^2}\right).\end{equation}
Let $V_{p}$ be the tangent cone of
$\widetilde{B}$
at $p$. Therefore,  $T_a$ transforms $\widetilde{B}$ and $\widehat{B}$ conformally to infinite cones
$\widetilde{V}$ and $\widehat{V}$ with vertices $0$, which are conjugate to
$V_{p}$. By a straightforward calculation, the Jacobi matrix $(\frac{\partial T_a}{\partial x})$ has the form 
\begin{equation}\label{eq-Jacobi}\left(\frac{\partial T_a}{\partial x}(x)\right)=\frac{2a^2}{a^2+2ax_1+|x|^2}O(x),
\end{equation}
where $O(x)$ is an orthogonal matrix. 

Let $\widetilde{v}$ be the solution of \eqref{eq-MainEq}-\eqref{eq-MainBoundary} for $\Omega = \widetilde{V}$.
Then, $\widetilde{v}^{\frac{4}{n-2}}(y_1,\cdots,y_n)\sum_{i=1}^{n}dy_i\otimes dy_i$ is
a complete metric with the constant scalar curvature $-n(n-1)$ on $\widetilde{V}$.  Hence,
$$\widetilde{u}^{\frac{4}{n-2}}(x)\sum_{i=1}^{n}dx_i\otimes dx_i =
T_a^{*}(\widetilde{v}^{\frac{4}{n-2}}\sum_{i=1}^{n}dy_i\otimes dy_i)$$
is a complete metric with the constant scalar curvature $-n(n-1)$ on $\widetilde{B}$. A straightforward 
calculation, with the help of \eqref{eq-Jacobi}, yields 
\begin{align}\label{eq-tilde-u}\widetilde{u}(x)
=\widetilde{v}\left(T_ax\right)
\left[\frac{2a^2}{a^2+2ax_1+|x|^2}\right]^{\frac{n-2}{2}}.\end{align}
Then, $\widetilde u$ is a solution of \eqref{eq-MainEq}-\eqref{eq-MainBoundary} for $\Omega = \widetilde{B}$.
In fact, by the scaling property of $\widetilde v$ and the explicit expression of $T_a$ in \eqref{eq-ConformalTransform}, 
we have 
\begin{align*}
\widetilde u(x)=\widetilde{v}(-a(a^2-|x|^2),2a^2 x_2,\cdots,2a^2 x_n)(2a^2)^{\frac{n-2}{2}}. 
\end{align*}
This expression will not be needed.

Similarly, let $\widehat{v}$ be the solution of \eqref{eq-MainEq}-\eqref{eq-MainBoundary}
for $\Omega = \widehat{V}$. Then,
\begin{align*}\widehat{u}(x)
=\widehat{v}\left(T_ax\right)
\left[\frac{2a^2}{a^2+2ax_1+|x|^2}\right]^{\frac{n-2}{2}},
\end{align*}
and $\widetilde{v}$ is a solution of \eqref{eq-MainEq}-\eqref{eq-MainBoundary} for $\Omega = \widehat{B}$.

Now we are ready to prove the main theorem of this section.

\begin{theorem}\label{opytimal estimate nd-nface}
Let $\Omega\subset \mathbb{R}^{n}$ be a bounded Lipschitz domain with
$x_0\in\partial\Omega$ and, for some integer $k\le n$,  $\partial\Omega$ in a neighborhood of
$x_0$ consists of
$k$ $C^{1,1}$-hypersurfaces $S_1, \cdots, S_k$ intersecting at $x_0$ with the property that the normal vectors
of $S_1, \cdots, S_k$ at $x_0$ are linearly independent. 
Suppose $ u \in
C^{\infty}(\Omega)$ is a solution of \eqref{eq-MainEq}-\eqref{eq-MainBoundary}
and $u_{V_{x_0}}$ is the corresponding solution in the tangent cone $V_{x_0}$.
Then,
for any $x$ close to $x_0$,
\begin{align}
\label{main-result} |u(x)-f_{V_{x_0}}(d_1,\cdots,d_k)|\leq Cu(x)|x-x_0|,
\end{align}
where  $d_i$ is the signed distance to $S_i$ with respect unit inner normal vectors,
$f_{V_{x_0}}$ is the function $u_{V_{x_0}}$ in
terms of $d_1, d_2,\cdots,d_k$ as in \eqref{eq-Solution-Cone-d-coordinates-nd2}, $C$ is a
positive constant depending only on $n$, $\sigma(P_i, \cdots, P_k)$ as defined in
\eqref{eq-LinearIndependence}
and the
$C^{1,1}$-norms of $S_1, \cdots, S_k$ near $x_0$.
\end{theorem}

\begin{proof} Let $\nu_1, \cdots, \nu_k$ be the inner unit normal vectors of $S_1, \cdots, S_k$ at $x_0$, respectively, 
and denote by $N$ the matrix $(\nu_1, \cdots, \nu_k)$ as in \eqref{eq-Definition-N}. 
We fix an $x\in \Omega$ near $x_0\in \partial\Omega$ and 
let $p_i$ be the point on $S_i$ with the least distance to $x$.
Denote by $\nu_{p_i}$ the inner unit normal vector of $S_i$ at $p_i$.
Let $ B_r(o_i) $ be the interior tangent ball of $S_i$ at $p_i$ with a radius $r$
and $B_r(o_i')$ be the exterior tangent ball of $S_i$ at $p_i$, with $r$ to be determined.
We point out that $ B_r(o_i) $ is not necessarily in $\Omega$ and that $B_r(o_i')$
is not necessarily outside $\Omega$. We now divide the proof in several steps. 

\smallskip 

{\it Step 1.} We construct two sets, one inside $\Omega$ and one containing $\Omega$. 
The set inside $\Omega$ is out of $ B_r(o_1), \cdots, B_r(o_k)$, while the set containing 
$\Omega$ is out of $ B_r(o_1')^c, \cdots, B_r(o_k')^c$.

We first prove that $ \partial B_r(o_1), \cdots, \partial B_r(o_k) $ intersect at at least two points 
with their distance bounded from below and that a similar result holds for $B_r(o_i')$. 

Denote by $P_{p_i}$ the tangent plane of $S_i$ at $p_i$. We assume $x= 0$, 
and
$$\nu_{p_{i}}= (\nu^{p_i}_{1},\cdots,\nu^{p_i}_{k},0_{\mathbb R^{n-k}}).$$
Consider the matrices
$$N_k =((\nu^{p_1}_{1},\cdots,\nu^{p_1}_{k})^{T},  \cdots, (\nu^{p_k}_{1},\cdots,\nu^{p_k}_{k})^{T}),$$
and 
$$N'=(\nu_{p_1},\cdots, \nu_{p_k}).$$ Note that
$N_k$ is a $k\times k$ matrix and $N'$ an $n\times k$ matrix.
We can parameterize
$P_{p_i}$ as $$\langle \nu_{p_i},y-p_i \rangle=0.$$ For $|x-x_0|$ small, we have
$$\|N_{k}^{-1}\|^{2} =\| (N'^{T} N')^{-1}\| \leq \frac{\|(N^T N)^{-1}\|}{1-C\|(N^T N)^{-1}|x-x_0|\|}<2\|(N^T N)^{-1}\|,$$
where $C$ is some positive constant depending only on $n$, $\|(N^T N)^{-1}\|$, and the
geometry of $\partial\Omega$.
We have $$|\langle \nu_{p_i},p_i\rangle|=\langle \nu_{p_i},p_i- x\rangle| \leq |x-x_0|.$$
Consider the vector $b=(\langle \nu_{p_1},p_1 \rangle,\cdots,\langle \nu_{p_k},p_k \rangle)^{T}\in \mathbb R^k$.
Then the system of linear equations
$$ \langle \nu_{p_i},y-p_i\rangle=0 \quad\text{for } i=1,...,k,$$
has a solution $p=((N_{k}^{T})^{-1}b,0_{\mathbb R^{n-k}})$ and $|p|\le C|x-x_0|.$

We view the graph of $\partial B_r(o_i)$ and  $\partial B_r(o_i')$ near $p$ as
small perturbations of $$\langle \nu_{p_i},y-p_i\rangle =0.$$
In other words, near 0,
$\partial B_r(o_i)$ is parametrized by
$$ \langle \nu_{p_i},y-p_i \rangle=\widetilde{g}_i(y),$$
and $ \partial B_r(o_i')$ is parametrized by
$$ \langle \nu_{p_i},y-p_i \rangle=\widehat{g}_i(y),$$
where we have, for $y \in B_{C|x-x_0|}$,
$$|\widetilde{g}_i(y)|,|\widehat{g}_i(y)|\le C_0(|x-x_0|+|y|)^2,$$
and
$$|D_y\widetilde{g}_i(y)|,|D_y\widehat{g}_i'(y)|\le C_0(|x-x_0|+|y|).$$
It is easy to verify that $G(y)=(N_{k}^{T})^{-1}(\widetilde{g}_{i}(y)+b)$
is a contraction mapping on $B_{C|x-x_0|}\bigcap$ $(\mathbb R^k\times\{0_{\mathbb R^{n-k}}\})$,
if $|x-x_0|$ is small.
Therefore, the  system of equations
$$y=G(y)$$ has a solution in $B_{C|x-x_0|}\bigcap (\mathbb R^k\times\{0_{\mathbb R^{n-k}}\})$.
Hence, there exists a point $\widetilde{p}$ such that
$$\widetilde{p} \in B_{C|x-x_0|}\bigcap (\mathbb R^k\times\{0_{\mathbb R^{n-k}}\})
\bigcap (\bigcap_{i=1}^k\partial B_r(o_i)).$$
Denote by $\widetilde{\nu}_{i}$ the inner normal vector of $B_r(o_i)$ at $\widetilde{p}$,
and by $\widetilde{N}$ the matrix $(\widetilde{\nu}_{1},...,\widetilde{\nu}_{k})$.
For $|x-x_0|$ small, we have
$$|\det(\widetilde{N}^{T}\widetilde{N})|=|\det N^{T}N ||
\det[I+(N^T N)^{-1}(\widetilde{N}^{T} \widetilde{N}-N^{T} N) ]  |\geq\frac14|\det N^T N |,$$
and
$$ \|(\widetilde{N}^{T}\widetilde{N})^{-1}\| \leq \frac{\|(N^T N)^{-1}\|}{1-C\|(N^T N)^{-1}|x-x_0|\|}
<2\|(N^T N)^{-1}\|.$$
By Lemma \ref{n ball inters}, we have
$$\bigcap_{i=1}^k \partial B_r(o_i) \bigcap (\mathbb R^k\times\{0_{\mathbb R^{n-k}}\})
=\{\widetilde{p},\widetilde{q}\},$$
and
$$|\widetilde{p}\widetilde{q}| >
r\frac{| \sqrt{\det \widetilde{N}^{T}\widetilde{N}}|}{2^{k-2}} > r\frac{\sqrt{\det N^T N}}{2^{k-1}}.$$
Similarly, there exists a point $\widehat{p}$,
$$\widehat{p} \in B_{C|x-x_0|}\bigcap(\mathbb R^k\times\{0_{\mathbb R^{n-k}}\})
\bigcap  (\bigcap_{i=1}^k\partial B_r(o_i')).$$
Denote by $\widehat{\nu}_{i}$ the unit outer normal vector of $B_r(o_i')$ at $\widehat{p}$,
and by $\widehat{N}$ the matrix $(\widehat{\nu}_{1},...,\widehat{\nu}_{k})$.  For $|x-x_0|$ small, we have
$$|\det\widehat{N}^{T}\widehat{N}|=|\det N^{T}N ||
\det(I+ (N^{T}N)^{-1}(\widehat{N}^{T}\widehat{N}-N^{T}N)   |\geq\frac14|\det N^{T}N |,$$
and
$$ \|(\widehat{N}^{T}\widehat{N})^{-1}\| \leq \frac{\|(N^{T}N)^{-1}\|}{1-C\|(N^{T}N)^{-1}|x-x_0|\|}
<2\|(N^{T}N)^{-1}\|.$$
Moreover,
$$\bigcap_{i=1}^k \partial B_r(o'_i)\bigcap (\mathbb R^k\times\{0_{\mathbb R^{n-k}}\}) =\{\widehat{p},\widehat{q}\},$$
and
$$|\widehat{p}\widehat{q}| > 2r\frac{\sqrt{\det\widehat{N}^{T} \widehat{N}}}{2^{k-1}}
> r\frac{\sqrt{\det N^{T}N}}{2^{k-1}}.$$

We now construct two sets, one inside $\Omega$ and another containing $\Omega$. 

Suppose, for some constant $R>0$,  $\Omega  \bigcap B_{R}(x_0)$ 
can be expressed by the union of some $\Omega_{(l_1,\cdots,l_k)}$, i.e., 
\begin{equation}\label{eq-Domaincomponet2}\Omega= \bigcup \Omega_{(l_1,\cdots, l_k)}, \end{equation}
where the union is over a finite set of vectors $(l_1,\cdots,l_k)$,  with $l_i=1$ or $-1$
for each $i=1,\cdots,k$. Refer to discussions in Section \ref{sec-Domains}. 

With $B_r(o_i)$ replacing $B_{1}^{n}(o_i)$,  
we can define $\widetilde{B}_{(l_1,\cdots,l_k)}$ as in \eqref{eq-definition-tildeB}, for any
$(l_1,\cdots,l_k)$,  with $l_i=1$ or $-1$
for each $i=1,\cdots,k$.
Then, we set 
\begin{equation}\label{eq-interiorballcomponet2}\widetilde{B}= \bigcup \widetilde{B}_{(l_1,\cdots,l_k)}, \end{equation}
where the union is over the same set of vectors 
$(l_1,\cdots,l_k)$ as in \eqref{eq-Domaincomponet2}. We note that $\widetilde B$ is a Lipschitz domain. 

Similarly, with $B_r(o_i')$ replacing  $B_{1}^{n}(o'_i)$, 
we can define $\widehat{B}_{(l_1,\cdots,l_k)}$ as in \eqref{eq-definition-hatB}, for any
$(l_1,\cdots,l_k)$,  with $l_i=1$ or $-1$
for each $i=1,\cdots,k$.
Then, we set 
\begin{equation}\label{eq-exteriorballcomponet2}\widehat{B}= \bigcup \widehat{B}_{(l_1,\cdots,l_k)}, \end{equation}
where the union is over the same set of vectors 
$(l_1,\cdots,l_k)$ as in \eqref{eq-Domaincomponet2}. Similarly, $\widehat B$ is a Lipschitz domain.

For some small constants $r$ and $r^*$ depending only on the
geometry of $\partial\Omega$, we have $\widetilde{B} \subseteq \Omega$
and $\Omega\subseteq \widehat{B}$ if $|x-x_0|\leq r^*$. 
We can check this by Lemma \ref{lemma-domain-relation}(ii). 
We point out that $r^*$ is relatively small compared with $r$. 
We also note that each ball
$B_r(o_i)$ is above the corresponding hypersurface $S_i$, although it is not necessarily in $\Omega$, 
and that each ball $B_r(o_i')$ is below the corresponding hypersurface $S_i$.

\smallskip

{\it Step 2.} We now compare the solution $u$ with the corresponding solution in the tangent cone. 

Let $\widetilde{u}$ be the solution of \eqref{eq-MainEq}-\eqref{eq-MainBoundary}
for $\Omega = \widetilde{B}$. By the maximum principle,
we have $u \leq\widetilde{u}$ in $\widetilde{B}$. In particular, we have $u(x) \leq\widetilde{u}(x)$.
By a translation and a rotation, we assume
$\widetilde{p}=(\widetilde{a},0,\cdots,0)$, and $\widetilde{q}=(-\widetilde{a},0,\cdots,0)$,
where
\begin{equation}\label{eq-LowerBound-tilde-a}
\widetilde a=\frac{|\widetilde{p}\widetilde{q}|}{2} >r \frac{|\sqrt{\det N^{T}N}|}{2^{k}}.\end{equation} 
Let $T_{\widetilde a}$ be the conformal transform given by \eqref{eq-ConformalTransform}, 
with $\widetilde a$ replacing $a$, and 
$\widetilde{v}$ be the solution of \eqref{eq-MainEq}-\eqref{eq-MainBoundary}
for $\Omega =\widetilde{V}=T_{\widetilde{a}}(\widetilde{B})$. 
Then, by \eqref{eq-Jacobi} and \eqref{eq-tilde-u}, 
the Jacobi matrix $(\frac{\partial T_{\widetilde a}}{\partial x})$ has the form 
$$\left(\frac{\partial T_{\widetilde a}}{\partial x}(x)\right)
=\frac{2\widetilde a^2}{\widetilde a^2+2\widetilde ax_1+|x|^2}O(x),
$$
where $O(x)$ is an orthogonal matrix, and 
\begin{align*}\widetilde{u}(x)
=\widetilde{v}(T_{\widetilde a}x)\left(\frac{2\widetilde{a}^2}{\widetilde a^2+2\widetilde ax_1+|x|^2}\right)^{\frac{n-2}{2}}.
\end{align*}
Let $\widetilde{P}_i$ be the hyperplane transformed from $ \partial B_r(o_i)$ by the map
$T_{\widetilde{a}}$.
Denote by $\widetilde{d}_i $ the signed distance from $T_{\widetilde a}x$ to $\widetilde{P}_i$.
We have
$$|x-\widetilde{p}|\leq C|x-x_0|.$$
By the explicit form of the Jacobi matrix $(\frac{\partial T_{\widetilde a}}{\partial x})$, we have, for $|x-x_0|$ small,
$$\widetilde{d}_i=
\left[\frac{2\widetilde{a}^2}{\widetilde a^2+2\widetilde ax_1+|x|^2}+O(|x-x_0|)\right]d_i.$$
Then, by using $\widetilde x=(\widetilde a, 0, \cdots, 0)$ at $\widetilde p$ and the lower bound of
$\widetilde a$ in \eqref{eq-LowerBound-tilde-a}, we obtain
$$\widetilde{d}_i=
\left(\frac{1}{2}+O(|x-x_0|)\right)d_i.$$
We now write
$\widetilde{v}=f_{\widetilde{V}}(\widetilde{d}_1,\cdots,\widetilde{d}_k)$
as in \eqref{eq-Solution-Cone-d-coordinates-nd2}.
By the scaling property \eqref{eq-scaling} and Lemma \ref{lemma-AnisotropicEstimate},
we have
$$f_{\widetilde{V}}(\widetilde{d}_1,\cdots,\widetilde{d}_k)
\leq f_{\widetilde{V}}(d_1,\cdots,d_k)\left(\frac{1}{2}-C|x-x_0|\right)^{-\frac{n-2}{2}}.$$
Therefore,
\begin{equation*}
u(x)\leq \widetilde{u}(x)=f_{\widetilde{V}}(\widetilde{d}_1,\cdots,\widetilde{d}_k)
\left[\frac{2\widetilde{a}^2}{\widetilde a^2+2\widetilde ax_1+|x|^2}\right]^{\frac{n-2}{2}}
\leq f_{\widetilde{V}}(d_1,\cdots,d_k)(1+C|x-x_0|).
\end{equation*}
Note $\|\widetilde{N}-N\|\leq C|x-x_0|$. By Lemma \ref{slu-k n1-n2 under tr-general},
we have $$ f_{\widetilde{V}}(d_1,\cdots,d_k) \leq  f_{V_{x_0}}(d_1,\cdots,d_k) (1+C|x-x_0|).$$
Therefore,
\begin{equation}\label{u leq-main}
u(x)\leq f_{V_{x_0}}(d_1,\cdots,d_k)(1+C|x-x_0|).
\end{equation}

Let $\widehat{u}$ be the solution of \eqref{eq-MainEq}-\eqref{eq-MainBoundary}
for $\Omega =\widehat{B} $. By the maximum principle, $\widehat{u} \leq u$ in $\Omega$.
In particular, we have $\widehat{u}(x) \leq u(x)$.
By a translation and a rotation, we assume
$\widehat{p}=(\widehat{a},0,\cdots,0)$ and $\widehat{q}=(-\widehat{a},0,\cdots,0)$,
where
$$\widehat{a}=\frac{|\widehat{p}\widehat{q}|}{2} >r \frac{|\sqrt{\det N^{T}N}|}{2^{k}}.$$ Similarly, we get
\begin{equation*}
u(x)\geq \widehat{u}(x)= \widehat{v}(T_{\widehat{a}}x)
\left[\frac{2\widehat{a}^2}{\widehat a^2+2\widehat ax_1+|x|^2}\right]^{\frac{n-2}{2} }
\geq f_{\widehat{V}}(d_1,\cdots,d_k)(1-C|x-x_0|),
\end{equation*}
where $T_{\widehat a}$ is the conformal transform 
given by \eqref{eq-ConformalTransform}, 
with $\widehat a$ replacing $a$, 
$\widehat{v}$ is the solution of \eqref{eq-MainEq}-\eqref{eq-MainBoundary} for
$\Omega =\widehat{V}=T_{\widehat{a}}(\widehat{B})$ and
$\widehat{v}=f_{\widehat{V}}(\widehat{d}_1,\cdots,\widehat{d}_k)$ as  in \eqref{eq-Solution-Cone-d-coordinates-nd2}.
By Lemma \ref{slu-k n1-n2 under tr-general}, we have
$$f_{\widehat{V}}(d_1,\cdots,d_k) \geq f_{V_{x_0}}(d_1,\cdots,d_k) (1-C|x-x_0|).$$
Therefore,
\begin{equation}\label{u geq-main}
u(x) \geq f_{V_{x_0}}(d_1,\cdots,d_k) (1-C|x-x_0|).
\end{equation}

Combining \eqref{u leq-main} and \eqref{u geq-main}, we complete the proof.
\end{proof}

Now we compare the proof of Theorem \ref{opytimal estimate nd-nface} and that of 
Theorem \ref{thrm-C-1,alpha-expansion}. In the proof of Theorem \ref{thrm-C-1,alpha-expansion}, we 
construct two balls, one inside $\Omega$ and one outside $\Omega$, and then compare the solution $u$ with 
the corresponding solution in the interior ball and the solution outside the 
exterior ball.  In the proof of Theorem \ref{opytimal estimate nd-nface}, we replace
the interior ball and the complement of the  
exterior ball by $\widetilde B$ and $\widehat B$, constructed from the interior tangent balls and 
the complements of the exterior tangent balls, respectively, and then compare the solution $u$ with 
the corresponding solutions in these two sets. The conformal structure of the equation plays an essential role
in the present proof. 

\smallskip 
We are ready to prove the main result in this paper. 

\begin{proof}[Proof of Theorem \ref{main reslut}]
We adopt the notations from Theorem \ref{opytimal estimate nd-nface} and its proof. Let $\Omega$ be 
bounded by $C^{1, 1}$ hypersurfaces $S_1, \cdots, S_k$ near $x_0\in\partial\Omega$ and the 
tangent cone $V_{x_0}$ of $\Omega$ at $x_0$ be bounded by $P_1, \cdots, P_k$, the tangent planes 
of $S_1, \cdots, S_k$ at $x_0$, respectively, with $\nu_1, \cdots, \nu_k$ the inner unit normal vectors. 

First, we consider the case $k=n$. Without loss of generality, we assume $x_0$ is 
the origin. For any $x$ sufficiently small, we define 
$$T_{\{S_i\}}x=(d_1(x), \cdots, d_n(x)),$$
where $d_i(x)$ is the signed distance from $x$ to $S_i$ with respect to $\nu_i$, $i=1, \cdots, n$. 
We emphasize that $T_{\{S_i\}}$ is defined in a full neighborhood of the origin instead of only in $\Omega$ 
and that the signed distance is used instead of its absolute value. Then, $T_{\{S_i\}}$ is $C^{1,1}$ near the origin and 
its Jacobi matrix at the origin is nonsingular by the linear independence of 
$\nu_1, \cdots, \nu_n$. Therefore, $T_{\{S_i\}}$ is a $C^{1,1}$-diffeormorphism 
in a neighborhood of the origin. We have a similar result for $T_{\{P_i\}}$, with $P_1, \cdots, P_n$ replacing 
$S_1, \cdots, S_n$. In fact, $T_{\{P_i\}}$ is a linear transform, since $P_1, \cdots, P_n$ are hyperplanes 
passing the origin. Then, the map $T=T^{-1}_{\{P_i\}}\circ T_{\{S_i\}}$ is a $C^{1,1}$-diffeomorphism 
near the origin and has the property that the signed distance from $x$ to $S_i$ is the same as that from $Tx$
to $P_i$, for $i=1, \cdots, n$.

Next, we consider the case $k< n$. 
We add hyperplanes $P_{k+1}, \cdots, P_n$ passing the point $x_0$ 
with their unit normal vectors $\nu_{k+1}, \cdots, \nu_n$ forming an orthonormal basis 
of the orthogonal complement of Span$\{\nu_1,..,\nu_k\}$, as in the proof of 
Theorem \ref{slu-k n1-n2 under tr-general}. 
We denote by $d_{j}$ the signed distance from $x$ to $P_{j}$ with respect to $\nu_j$, 
$j=k+1,\cdots,n.$ 
Then we can construct the map $T$ as in the case $k=n$.
\end{proof} 

We point out that the assumption $k\le n$ plays an essential role in the proof of 
the optimal estimate \eqref{main-estimate} as stated in 
Theorem \ref{main reslut}. For example, 
$k=n$ is needed crucially in the identification of the points by their 
signed distances to $n$ $C^{1,1}$-hyperplanes as in Section \ref{sec-Domains}
and hence in the construction of 
the transform $T$. 
Moreover, the assumption $k\le n$ is used to find an intersect of $k$ spheres near $x_0$ 
as in Step 1 of the proof of Theorem \ref{opytimal estimate nd-nface}
and to obtain a lower bound of the distance between this intersect and another intersect away from $x_0$ as 
in Lemma \ref{n ball inters}. 

\section{Solutions in General Singular Domains}\label{Slu-GerSingularDomains}

In this section, we study asymptotic
expansions near singular points for more general domains. Specifically, 
we allow $k>n$ in Theorem \ref{opytimal estimate nd-nface} and derive optimal estimates  
for points strictly located inside the tangent cone. The discussion again relies essentially
on the conformal invariance of the equation \eqref{eq-MainEq}.

Let $\Omega $ be a bounded Lipschitz domain. We fix a point $x_0 \in \partial \Omega$ and
assume that it is the origin  such that, for some $R>0$, 
$$\Omega\cap B_{2R}(x_{0})
=\{x\in B_{2R}(x_0): x_{n}>f(x')\},$$ 
for some Lipschitz function $f$ on $B_{2R}'$ with $f(0)=0$. 
Then, there exists a finite circular cone
$V_{\theta_{0}}$, with $x_{0}$ as its vertex, $x_n$-axis as the axis of the
cone, an apex angle $2\theta_{0}$ and a height $h$, 
such that 
\begin{equation}\label{eq-FiniteCone}
V_{\theta_{0}}\subseteq \overline{\Omega},\quad -V_{\theta_{0}}\subseteq \Omega^{c},\end{equation}
where $-V_{\theta_{0}}$ is the reflection of $V_{\theta_{0}}$ about $x_{n}=0$ and 
$\theta_{0}$ and $h$ are constants depending only on the geometry of $\partial \Omega$.

Fix an integer $k\ge 2$. Recall the domain $\Omega$ introduced in 
Definition \ref{def-Domain-lipschitz} and the subsequent discussion of its decomposition 
and its tangent cones.

We now prove an important property concerning tangent cones.

\begin{lemma}\label{lemma-TangentCones}
For some point $x_0\in\partial\Omega$, let 
$\Omega$ be a bounded Lipschitz domain bounded 
by $k$ $C^{1,1}$-hypersurfaces $S_1, \cdots, S_k$ near $x_0$ 
as in Definition \ref{def-Domain-lipschitz}
and let
$V_{x_{0}}$ be the tangent cone of $\Omega$ at $x_0$ in the setting following 
Definition \ref{def-Domain-lipschitz}. 
Then,  for any small $r$,
$$\left(V_{x_{0}}+\frac{Mr^{2}}{\sin\theta_{0}}e_n\right)
\bigcap B_{r} (x_{0})\subset \Omega,$$
and
$$\Omega \bigcap B_{r} (x_{0})
\subset
V_{x_{0}}-\frac{Mr^{2}}{\sin\theta_{0}}e_n,$$
where $M$ is the maximum of the $C^{1,1}$-norms of $S_1,\cdots, S_k$ near $x_0$
and $\theta_0$ is introduced for \eqref{eq-FiniteCone}. 
\end{lemma}

\begin{proof} We assume $x_0$ is the origin. 
Let $P_{i}'$ and $P_{i}''$ be the hyperplanes transformed from $P_{i}$ in the direction of
$\nu_i$ and $-\nu_i$ by a distance of $r^{2 }$, respectively. 
Assume that the hyperplane $P_i$ is expressed by 
$x_{n}=L_{i}(x')$ for some linear function $L_i$ and that 
the hypersurface $S_i$ near $x_0$ is expressed by $x_{n}=f_{i}(x')$
for some $C^{1,1}$-function $f_i$, 
for each $i=1,2,\cdots,k$. Note, by adjusting $M$, 
$$L_{i}(x')-M|x'|^{2}\leq f_{i}(x')\leq L_{i}(x')+M|x'|^{2}.$$
By 
$$M\left((M^{-1}r)^{\frac{1}{2}}\right)^{2}=r,$$
we have
\begin{equation}\label{domain contain cone}\left(V_{x_{0}}+\frac{r}{\sin\theta_{0}}e_n\right)
\bigcap B_{(M^{-1}r)^{\frac{1}{2}}} (x_{0})\subset \Omega,\end{equation}
and
\begin{equation}\label{cone contain domain}\left(\Omega \bigcap B_{(M^{-1}r)^{\frac{1}{2}}} (x_{0})\right)
\subset V_{x_{0}}-\frac{r}{\sin\theta_{0}}e_n.\end{equation}
In fact, take any $x \in \left(V_{x_{0}}+\frac{r}{\sin\theta_{0}}e_n\right)\bigcap B_{(M^{-1}r)^{\frac{1}{2}}} (x_{0})$.
Then, for some $(l_1,\cdots,l_k)$ in \eqref{eq-Domaincomponet2},
$$x\in\left(V_{(l_1,\cdots,l_k)}+\frac{r}{\sin\theta_{0}}e_n\right)\bigcap B_{(M^{-1}r)^{\frac{1}{2}}} (x_{0}),$$
and, as a consequence, for some $(l_1',\cdots,l_k') \geq (l_1,\cdots,l_k)$,
$$x\in \left(\Omega_{(l_1',\cdots,l_k')}+\frac{r}{\sin\theta_{0}}e_n\right)\bigcap B_{(M^{-1}r)^{\frac{1}{2}}} (x_{0}),$$  
since the graph of $ P_i+\frac{r}{\sin\theta_{0}}e_n$ is above the graph of $S_i$ 
in $B'_{(M^{-1}r)^{\frac{1}{2}}} (x_{0})$. 
Hence, $x \in \Omega$. This proves \eqref{domain contain cone}. 
A similar argument yields \eqref{cone contain domain}. 
We have the desired result by renaming radii in \eqref{domain contain cone}
and \eqref{cone contain domain}.
\end{proof}

Lemma \ref{lemma-TangentCones} asserts the following statement:
In a neighborhood of $x_{0}$, if $V_{x_{0}}$ is translated in the direction of
$e_n$ by an appropriate distance,
its graph is above the graph of $\partial\Omega$,  and if $V_{x_{0}}$
is translated in the direction of $-e_n$ by
the same distance,  its graph is below the graph of $\partial\Omega$.

We now prove the main result in this section. 

\begin{theorem}\label{opytimal estimate general domain a b}
For some point $x_0\in\partial\Omega$, let 
$\Omega$ be a bounded Lipschitz domain bounded 
by $k$ $C^{1,1}$-hypersurfaces $S_1, \cdots, S_k$ near $x_0$ 
as in Definition \ref{def-Domain-lipschitz}
and let
$V_{x_{0}}$ be the tangent cone of $\Omega$ at $x_0$. 
Suppose that $ u \in
C^{\infty}(\Omega)$ is a solution of \eqref{eq-MainEq}-\eqref{eq-MainBoundary} and that 
$v$ is the corresponding solution in $V_{x_0}$.
Then, for any $\delta >0$ and 
any $x\in\Omega$ close to $x_0$ with $\textrm{dist}(x,\partial \Omega) >\delta |x-x_0|$,
\begin{align}
\label{main-result2} |u(x)-v(x)|\leq C \delta^{-1}u(x)|x-x_0|,
\end{align}
where $C$ is a positive constant depending only on $n$, $R$, $\theta_0$
and the
$C^{1,1}$-norms of hypersurfaces $S_1, \cdots, S_k$ near $x_0$.
\end{theorem}

\begin{proof}
We fix an $x\in \Omega$ near $x_0\in \partial\Omega$ with 
$\textrm{dist}(x,\partial \Omega) >\delta |x-x_0|$. We denote by $\nu_i$ the
interior unit normal vector of $S_i$ at $x_0$. Then, by \eqref{eq-FiniteCone}, 
\begin{equation}\label{cone-confition}\langle\nu_i, e_n\rangle > \sin \theta_0.\end{equation}
For some constant $r>0$, set 
$$r_i =\frac{r}{\langle\nu_i, e_n\rangle},$$
and 
$$\aligned 
\widetilde{B}_i&=B_{r_i}(x_0+r_i \nu_i),\\
\widehat{B}_i&=B_{r_i}(x_0-r_i \nu_i).\endaligned$$ 
By \eqref{cone-confition}, all $r_i$ are comparable.
Then, we have 
$$x_0,x_0+2re_n\in\bigcap_{i=1}^{k} \partial \widetilde{B}_i,$$ 
and 
$$x_0,x_0-2re_n\in\bigcap_{i=1}^{k}\partial  \widehat{B}_i.$$
For some constant $r$ depending only on $R$ and the $C^{1,1}$-norm of $S_i$,
we note that each ball
$\widetilde{B}_i$ is above the corresponding hypersurface $S_i$, although it is not necessarily in $\Omega$,
and that each ball $\widehat{B}_i$ is below the corresponding hypersurface $S_i$.

For some constant $R>0$, $\Omega  \bigcap B_{R}(x_0)$
can be expressed by the union of some $\Omega_{(l_1,\cdots,l_k)}$, i.e.,
\begin{equation}\label{eq-Domaincomponet3}\Omega= \bigcup \Omega_{(l_1,\cdots, l_k)}, \end{equation}
where the union is over a finite set of vectors $(l_1,\cdots,l_k)$,  with $l_i=1$ or $-1$
for each $i=1,\cdots,k$.

With $\widetilde{B}_i$ replacing $B_{1}^{n}(o_i)$,
we can define $\widetilde{B}_{(l_1,\cdots,l_k)}$ as in \eqref{eq-definition-tildeB}, for any
$(l_1,\cdots,l_k)$  with $l_i=1$ or $-1$
for each $i=1,\cdots,k$.
Then, we set
\begin{equation}\label{eq-interiorballcomponet3}\widetilde{B}= \bigcup \widetilde{B}_{(l_1,\cdots,l_k)}, 
\end{equation}
where the union is over the same set of vectors
$(l_1,\cdots,l_k)$ as in \eqref{eq-Domaincomponet3}. We note that $\widetilde B$ is a Lipschitz domain.

Similarly, with $\widehat{B}_i$ replacing  $B_{1}^{n}(o'_i)$,
we can define $\widehat{B}_{(l_1,\cdots,l_k)}$ as in \eqref{eq-definition-hatB}, for any
$(l_1,\cdots,l_k)$  with $l_i=1$ or $-1$
for each $i=1,\cdots,k$.
Then, we set
\begin{equation}\label{eq-exteriorballcomponet3}\widehat{B}= \bigcup \widehat{B}_{(l_1,\cdots,l_k)}, \end{equation}
where the union is over the same set of vectors
$(l_1,\cdots,l_k)$ as in \eqref{eq-Domaincomponet3}. Similarly, $\widehat B$ is a Lipschitz domain.

For some small constants $r$ and $r^*$ depending only on the
geometry of $\partial\Omega$, we have 
$\widetilde{B} \subseteq \Omega$ and $ \Omega\subseteq \widehat{B}$ if $|x-x_0|\leq r^*$. 
By Lemma \ref{lemma-TangentCones}, we have, for $x \in \widetilde{B}$ with 
$|x-x_0|$ small,
$$\textrm{dist}( x, \partial\widetilde{B})\geq \textrm{dist}( x,\partial \Omega) -C|x-x_0|^2,$$ and 
$$\textrm{dist}( x, \partial\widehat{B})\leq \textrm{dist}( x,\partial \Omega) +C|x-x_0|^2.$$ 

Let $\widetilde{u}$ be the solution of \eqref{eq-MainEq}-\eqref{eq-MainBoundary}
for $\Omega = \widetilde{B}$. By the maximum principle,
we have $u \leq\widetilde{u}$ in $\widetilde{B}$. In particular, we have $u(x) \leq\widetilde{u}(x)$.

By a translation and a rotation, we assume
$x_0=(r,0,\cdots,0)$, and $x_0+2re_n=(-r,0,\cdots,0)$.
Let $T_{r}$ be the conformal transform given by \eqref{eq-ConformalTransform}, 
with $r$ replacing $a$, and let $\widetilde{v}$ be the solution of 
\eqref{eq-MainEq}-\eqref{eq-MainBoundary}
for $\Omega =\widetilde{V}=T_{r}(\widetilde{B})$. 
Then, $T_{r}(\widetilde{B})=V_{x_0}-re_1$  and 
\begin{align*}\widetilde{u}(x)
=\widetilde{v}(T_{r}x)\left(\frac{2r^2}{r^2+2rx_1+|x|^2}\right)^{\frac{n-2}{2}}.
\end{align*}
Note
$$\left(\frac{\partial T_{r}}{\partial x}(x_0)\right)
=\frac{1}{2}I_{n\times n}.
$$ 
Hence, 
\begin{align*}
|(T_{r}x-T_rx_0)-\frac{1}{2}(x-x_0)|=O(|x-x_0|^2 ).
\end{align*}
Note, for $|x-x_0|$ small,  
$$\textrm{dist}(x,\partial V_{x_0})>\frac{\delta}{2}|x-x_0|.$$ 
Then, by Lemma \ref{lemma-GradientEstimates}, we have
\begin{equation*}
|\widetilde{v}(T_{r}x)-2^{\frac{n-2}{2}}v(x)|\leq C\delta^{-1}v(x)|x-x_0|.\end{equation*}
Therefore,
\begin{equation}\label{u leq-main2}
u(x)\leq \widetilde{u}(x) = \widetilde{v}(T_{r}x)
\left[\frac{2r^2}{r^2+2rx_1+|x|^2}\right]^{\frac{n-2}{2}}
\leq v(x)(1+C\delta^{-1}|x-x_0|).
\end{equation}

Similarly, we can prove
\begin{equation}\label{u geq-main2}
u(x) \geq  v(x) (1-C\delta^{-1}|x-x_0|).
\end{equation}
Combining \eqref{u leq-main2} and \eqref{u geq-main2}, we have the desired result.
\end{proof}

We note that the constant $C$ in Theorem \ref{opytimal estimate general domain a b} 
depends on the $C^{1,1}$-norms of $S_1,\cdots, S_k$ but independent of the number of hypersurfaces. 
Appropriately modified, Theorem \ref{opytimal estimate general domain a b} allows us to discuss asymptotic
expansions near singular points of other types. For example, if $V_{x_0}$ 
has an isolated singularity (at its vertex), we can approximate $V_{x_0}$ by a sequence of cones $V_k$ such that 
the number of faces of $V_k$ approaches the infinity as $k\to\infty$
and the $C^{1,1}$-norms of $\partial V_k$ remain uniformly bounded. 
As a consequence, a result similar as Theorem \ref{opytimal estimate general domain a b} holds 
for this class of domains. 

\end{document}